\documentclass[12pt]{article}

\usepackage{latexsym,amsmath,amscd,amssymb,graphics,float}
\usepackage{enumerate}

\usepackage{graphicx,lscape}

\usepackage[colorlinks]{hyperref}
\usepackage{url}

\begin{document}

\newenvironment{proof}[1][Proof]{\textbf{#1.} }{\ \rule{0.5em}{0.5em}}

\newtheorem{theorem}{Theorem}[section]
\newtheorem{definition}[theorem]{Definition}
\newtheorem{lemma}[theorem]{Lemma}
\newtheorem{remark}[theorem]{Remark}
\newtheorem{proposition}[theorem]{Proposition}
\newtheorem{corollary}[theorem]{Corollary}
\newtheorem{example}[theorem]{Example}

\numberwithin{equation}{section}
\newcommand{\ep}{\varepsilon}
\newcommand{\R}{{\mathbb  R}}
\newcommand\C{{\mathbb  C}}
\newcommand\Q{{\mathbb Q}}
\newcommand\Z{{\mathbb Z}}
\newcommand{\N}{{\mathbb N}}

\newcommand{\bfi}{\bfseries\itshape}

\newsavebox{\savepar}
\newenvironment{boxit}{\begin{lrbox}{\savepar}
\begin{minipage}[b]{15.5cm}}{\end{minipage}\end{lrbox}
\fbox{\usebox{\savepar}}}

\title{{\bf Natural geometric Fourier transforms and the associated fractional Laplacian}}
\author{R\u{a}zvan M. Tudoran}

\date{}
\maketitle \makeatother

\begin{abstract}
To each arbitrary given general geometric structure on $\mathbb{R}^{n}$, we associate a pair of compatible Fourier transforms, that prove to appear naturally in the framework of Poisson's summation formula for full lattices. We study their properties and the compatibility with the classical $n-$dimensional Fourier transform. In the case of a positive definite geometric structure, we show that these geometric Fourier transforms induce a geometric fractional Laplacian, with properties similar to those of the classical fractional Laplacian.
\end{abstract}

\medskip



\section{Introduction}
\label{section:one}

The aim of this article is to introduce and study certain pairs of geometric $n-$dimensional Fourier transforms, naturally associated to each arbitrary given geometric structure on $\mathbb{R}^{n}$, where by geometric structure we mean any nondegenerate real bilinear form on $\mathbb{R}^{n}$. Those transforms are natural geometric generalizations of the classical Fourier transform, since when the geometric structure is the canonical Euclidean inner product of $\mathbb{R}^{n}$, the associated pair of Fourier transforms consists of two copies of the classical Fourier transform. In the case of a general geometric structure $b:\mathbb{R}^{n}\times\mathbb{R}^{n}\rightarrow \mathbb{R}$ (which needs not have any symmetry--like properties) the associated pair of Fourier transforms consists of the \textit{left Fourier transform} (generated by $e^{-2\pi i b(\xi,\cdot)}$, $\xi\in\mathbb{R}^{n}$), and the \textit{right Fourier transform} (generated by $e^{-2\pi i b(\cdot,\xi)}$, $\xi\in\mathbb{R}^{n}$). Although for any arbitrary fixed $\xi\in\mathbb{R}^{n}$, each of the functions $$e^{-2\pi i\langle\cdot,\xi\rangle}, e^{-2\pi i b(\xi,\cdot)}, e^{-2\pi i b(\cdot,\xi)} :\mathbb{R}^{n}\rightarrow \mathbb{S}^{1},$$ is a character of the additive group $(\mathbb{R}^{n},+)$, and consequently, from the point of view of abstract Fourier analysis, there are no essential differences between the classical Fourier transform and the left/right Fourier transforms, the specific differences we shall focus on, are related to the \textbf{compatibility between these transforms and the group of automorphisms of $\mathbb{R}^{n}$ which preserve the associated geometric structure} (i.e., the orthogonal group, $O_{n}(\mathbb{R})$, associated to the canonical Euclidean inner product $\langle\cdot,\cdot\rangle$, in the case of the classical Fourier transform, and the group $G_b :=\{A\in \operatorname{End}(\mathbb{R}^{n}): b(A\mathbf{x},A\mathbf{y})=b(\mathbf{x},\mathbf{y}),\forall \mathbf{x},\mathbf{y}\in\mathbb{R}^{n}\}$, in the case of the left/right Fourier transforms). The compatibility between these Fourier transforms and the geometry generated by the geometric structure $b$, makes them a natural choice when deal with applications requiring specific $G_b-$invariant properties (e.g., image processing with focus on certain ($G_b-$)invariant features).

The structure of the article is the following. In the second section we recall the definition together with some useful properties of the gradient and also the Laplace operator, generated by a general geometric structure on $\mathbb{R}^{m}$. The third chapter is dedicated entirely to the study of the left/right $n-$dimensional Fourier transforms associated to a general geometric structure on $\mathbb{R}^{n}$. Here we define these transforms, analyze their main properties, and study their compatibility with the classical $n-$dimensional Fourier transform. In the fourth section we show that the left/right Fourier transforms appear naturally in the context of Poisson's summation formula for full lattices. In the last section we define a fractional Laplacian naturally associated to each pair $(b,s)$ consisting of a positive definite geometric structure $b$, and a real number $s\in (0,1)$. Here we prove that the geometric fractional Laplacian, has natural properties similar to those of the classical fractional Laplacian, thus making it a good choice to analyze similar problems in the appropriate geometric framework.

\section{Natural geometric structures on $\mathbb{R}^n$ and the associated gradient and Laplace operators}

In this section we briefly recall the definition together with some useful properties of the gradient and also the Laplace operator, generated by a general geometric structure on $\mathbb{R}^{m}$. Before doing this, let us notice that the classical geometries on $\mathbb{R}^n$ (i.e., Euclidean, Minkowski, pseudo-Euclidean, symplectic) are each generated by some nondegenerate real bilinear form, together with some specific additional properties. Thus, a general geometric setting that incorporates all those geometries of $\mathbb{R}^{n}$, was recently introduced in \cite{TDR}, and simply consits of pairs $(\mathbb{R}^{n},b)$, where $b$ is a nondegenerate real bilinear form, bearing the name of \textit{geometric structure}. Sometimes, the shorthand notation $\mathbb{R}^{n}_{b}$ is also used to denote the pair $(\mathbb{R}^{n},b)$.

Let us recall now from \cite{TDR}, an essential result for the rest of the article. More exactly, given an arbitrary fixed geometric structure $b$ on $\mathbb{R}^{n}$, there exists a \textit{unique} invertible linear map $B\in\operatorname{Aut}(\mathbb{R}^{n}):=\{A:\mathbb{R}^{n}\rightarrow\mathbb{R}^{n} \mid A ~\text{is}~\mathbb{R}-\text{linear and invertible}\}$ such that 
\begin{equation}\label{relb}
\langle \mathbf{x},\mathbf{y} \rangle = b(\mathbf{x}, B\mathbf{y}), ~~\forall \mathbf{x},\mathbf{y}\in\mathbb{R}^{n},
\end{equation}
where $\langle \cdot,\cdot \rangle$ stands for the canonical Euclidean inner product on $\mathbb{R}^{n}$. The pair $(b,B)$ is refer to as the \textit{geometric pair} on $\mathbb{R}^{n}$ generated by the geometric structure $b$.

Another useful result from \cite{TDR} states that given a geometric structure $b$ on $\mathbb{R}^n$, for any operator $A\in\operatorname{End}(\mathbb{R})$ one can define two natural operators compatible with the geometric structure, called the \textit{left-adjoint}, respectively the \textit{right-adjoint} of $A$ with respect to the geometric structure $b$, uniquely defined by the relations:
\begin{equation*}
b(A^{\star_{L}}\mathbf{x}, \mathbf{y})= b(\mathbf{x}, A\mathbf{y}), ~~\forall \mathbf{x},\mathbf{y}\in\mathbb{R}^{n},
\end{equation*}
\begin{equation*}
b(A\mathbf{x}, \mathbf{y})= b(\mathbf{x}, A^{\star_{R}}\mathbf{y}), ~~\forall \mathbf{x},\mathbf{y}\in\mathbb{R}^{n}.
\end{equation*}
As shown in \cite{TDR}, in terms of the geometric pair $(b,B)$, the left/right--adjoint operators of $A\in\operatorname{End}(\mathbb{R}^{n})$ are given by $A^{\star_{L}}=B^{\top}A^{\top}B^{-\top}$, and respectively $A^{\star_{R}}=B A^{\top}B^{-1}$, where $\top$ denotes the classical adjoint with respect to the canonical Euclidean inner product on $\mathbb{R}^n$.

Now, given a geometric structure $b$ on $\mathbb{R}^n$, there exists a natural subgroup of the automorphisms group $\operatorname{Aut}(\mathbb{R}^{n})$, consisting of those linear mappings which preserve the geometric structure $b$, i.e., 
\begin{align}\label{Gb}
G_b :=\{A\in \operatorname{End}(\mathbb{R}^{n}): b(A\mathbf{x},A\mathbf{y})=b(\mathbf{x},\mathbf{y}),\forall \mathbf{x},\mathbf{y}\in\mathbb{R}^{n}\}.
\end{align}
Recall from \cite{TDR} that in terms of left/right-adjoint operators, the group $G_b$ is given by $$G_b =\{ A\in \operatorname{End}(\mathbb{R}^{n}):~ A^{\star_L} A =I_{d}\}=\{ A\in \operatorname{End}(\mathbb{R}^{n}):~ A^{\star_R} A=I_{d}\},$$
and moreover, using the geometric pair $(b,B)$, 
\begin{equation}\label{ecimp2}
G_b=\{A\in\operatorname{End}(\mathbb{R}^{n}):~ A B A^{\top}=B\}.
\end{equation}
In \cite{TDR} it is proved that $G_b$ is a Lie group, whose Lie algebra, $\mathfrak{g}_{b}$, is given by 
\begin{align*}
\mathfrak{g}_{b}&=\{A\in \operatorname{End}(\mathbb{R}^{n}): b(A\mathbf{x},\mathbf{y})=-b(\mathbf{x},A\mathbf{y}),\forall \mathbf{x},\mathbf{y}\in\mathbb{R}^{n}\}\\
&=\{A\in \operatorname{End}(\mathbb{R}^{n}): A^{\star_L}=-A\}=\{A\in \operatorname{End}(\mathbb{R}^{n}): A^{\star_R}=-A\},
\end{align*}
or equivalently, in terms of the geometric pair $(b,B)$, $$\mathfrak{g}_{b}=\{A\in\operatorname{End}(\mathbb{R}^{n}):~ AB=-BA^{\top}\}.$$
Notice that when $b$ is the canonical inner product on $\mathbb{R}^n$, then $G_b=O_n(\mathbb{R})$ is the orthogonal group; when $b$ is the Minkowski inner product on $\mathbb{R}^{n}$, then $G_b=O_{n-1,1}(\mathbb{R})$ is the (full) Lorentz group; when $b$ is the canonical symplectic structure on $\mathbb{R}^{n}$, $n\in 2\mathbb{N}$, then $G_b=Sp_{n}(\mathbb{R})$ is the symplectic group.

Let us recall now from \cite{TDR} the definition of \textit{left/right--gradient vector fields}, naturally associated to any given geometric structure on $\mathbb{R}^n$. More precisely, given a geometric structure $b$ on $\mathbb{R}^n$ and $\Omega\subseteq\mathbb{R}^{n}$ an open set, then for each $f\in\mathcal{C}^{1}(\Omega,\mathbb{R})$, the \textit{left--gradient} of $f$ (denoted by $\nabla^{L}_{b}f$) is the vector field uniquely defined by the relation
$$
b(\nabla^{L}_{b}f(\mathbf{x}),\mathbf{v})=\mathrm{d}f(\mathbf{x})\cdot \mathbf{v}, ~ \forall \mathbf{x}\in \Omega,\forall\mathbf{v}\in T_{\mathbf{x}}\Omega\cong \mathbb{R}^{n}.
$$
Similarly, the \textit{right--gradient} of $f$ (denoted by $\nabla^{R}_{b}f$) is the vector field uniquely defined by the relation
$$
b(\mathbf{v},\nabla^{R}_{b}f(\mathbf{x}))=\mathrm{d}f(\mathbf{x})\cdot \mathbf{v}, ~ \forall \mathbf{x}\in \Omega,\forall\mathbf{v}\in T_{\mathbf{x}}\Omega\cong \mathbb{R}^{n}.
$$
The relation between $\nabla^{L}_{b}$, $\nabla^{R}_{b}$ in terms of the geometric pair $(b,B)$ associated to some general geometric structure $b$ is provided by the following result from \cite{TDR} which states that for any fixed open set $\Omega\subseteq\mathbb{R}^{n}$, and for every $f\in\mathcal{C}^{1}(\Omega,\mathbb{R})$, we have that $\nabla^{L}_{b}f=B^{\top}\nabla f$, $\nabla^{R}_{b}f=B \nabla f$, and $\nabla^{L}_{b}f=B^{\top}B^{-1}\nabla^{R}_{b}f$, where $\nabla$ stands for the classical gradient operator.

Next we recall from \cite{TDR} the definition of the \textit{b--Laplace operator} naturally associated to an arbitrary given geometric structure on $\mathbb{R}^n$. This operator generalizes the classical Laplace operator (in the case when the geometric structure is Euclidean) and also the d'Alembert operator (in the case when the geometric structure is Minkowski). 
\begin{definition}\cite{TDR}\label{lap0}
Let $b$ be a geometric structure on $\mathbb{R}^n$ and $\Omega\subseteq\mathbb{R}^{n}$ be an open set. Then for each $f\in\mathcal{C}^{2}(\Omega,\mathbb{R})$, the left--Laplacian of $f$, is given by $\Delta ^{L}_{b}f:=\operatorname{div}(\nabla ^{L}_{b}f)$, and similarly, the right--Laplacian of $f$, is given by $\Delta ^{R}_{b}f:=\operatorname{div}(\nabla ^{R}_{b}f)$, where $\operatorname{div}$ is the classical divergence operator.
\end{definition}
An important result from \cite{TDR} states that for any geometric structure $b$, the left and right Laplace operators always coincide, thereby introducing a unique geometric Laplace operator naturally associated to the geometric structure $b$, denoted by $\Delta_b$ and called the \textit{$b-$Laplace operator} or the \textit{Laplace operator generated by the geometric structure $b$}. Moreover, in terms of geometric pairs, the following formula holds. 
\begin{remark}\cite{TDR}\label{deltaa}
If $(b,B)$ is the geometric pair associated to a general geometric structure $b$ on $\mathbb{R}^{n}$, and $b_{kl}$, $k,l\in\{1,\dots,n\}$, are the entries of the matrix representing $B$ with respect to canonical basis, then $\Delta_b =\sum_{1\leq k,l \leq n} b_{kl}~\partial_{k}\partial_{l}$.
\end{remark}
Notice that if $b$ is the canonical inner product on $\mathbb{R}^n$ then $\Delta_{b}=\Delta$, where $\Delta$ is the classical Laplace operator; if $b$ is the Minkowski inner product on $\mathbb{R}^n$ then $\Delta_{b}=\square$, where $\square$ is the d'Alembert operator; if $b$ is a skew--symmetric geometric structure on $\mathbb{R}^n$, $n\in 2\mathbb{N}$, (i.e., $b$ is a symplectic structure) and $\Omega\subseteq\mathbb{R}^{n}$ is an open set, then $\Delta_{b}f\equiv  0, ~\forall f\in\mathcal{C}^{2}(\Omega,\mathbb{R})$.

In the following, we recall some useful geometric properties of the $b-$Laplace operator. In order to do that, we shall need the following action of the group $G_b$ on scalar functions defined on $\mathbb{R}^{n}$, given by 
\begin{equation}\label{act1}
(\tau_{A} f)(\mathbf{x}):=f(A^{-1}\mathbf{x}), \forall\mathbf{x}\in\mathbb{R}^{n}, ~\forall A\in G_b.
\end{equation}
Let us point out that for any given subgroup $H\leq G_b$, a scalar function $f:\mathbb{R}^{n}\rightarrow \mathbb{R}$ is $H-$invariant if and only if $\tau_{A} f=f, ~\forall A\in H$.

We conclude this section by recalling from \cite{TDR} the following result, which proves the $G_b -$equivariance of the $b-$Laplace operator and the associated iterates of arbitrary order.
\begin{theorem}\cite{TDR}\label{timpor}
Let $b$ be a geometric structure on $\mathbb{R}^n$ and $f\in\mathcal{C}^{\infty}(\mathbb{R}^{n},\mathbb{R})$. Then for any $m\in\mathbb{N}$, the following relation holds true
\begin{equation*}
\Delta_{b}^{m}(\tau_{A} f)=\tau_{A}(\Delta_{b}^{m}f), ~\forall A\in G_b,
\end{equation*}
where $\Delta_{b}^{m}:=\Delta_{b}\circ\dots\circ\Delta_{b}$, $m$ times, and $\Delta_{b}^{0}:={\operatorname{Id}}_{\mathcal{C}^{\infty}(\mathbb{R}^{n},\mathbb{R})}$.
\end{theorem}
A direct consequence of Theorem \ref{timpor}, is the following result, which provides a method to generate invariant functions, starting from a given one. More precisely, given $b$ a geometric structure on $\mathbb{R}^{n}$, $H$ a subgroup of $G_b$, and $f\in\mathcal{C}^{\infty}(\mathbb{R}^{n},\mathbb{R})$ an $H-$invariant function, then for any $m\in\mathbb{N}$, the function $\Delta_{b}^{m}f$ is also $H-$invariant. Obviously, the result still holds true, even if the function $f$ is only defined on some $H-$invariant open subset of $\mathbb{R}^{n}$.

\section{The $n-$dimensional left/right Fourier transforms}

The purpose of the section is to introduce and analyze the main protagonists of this article, namely, the left/right Fourier transforms associated to a general geometric structure on $\mathbb{R}^{n}$. These Fourier transforms are natural geometric generalizations of the classical Fourier transform, since when the geometric structure is the canonical Euclidean inner product on $\mathbb{R}^{n}$, the left and right Fourier transforms coincide and equal the classical $n-$dimensional Fourier transform. 

Since the most suitable environment to provide as many properties of the Fourier transform is a certain function space, let us start this section by recalling the definition of this space, namely, the \textit{Schwartz space} or the space of \textit{rapidly decreasing smooth functions on $\mathbb{R}^{n}$}, i.e.,
\begin{align*}
\mathcal{S}(\mathbb{R}^{n})&=\{f\in\mathcal{C}^{\infty}(\mathbb{R}^{n},\mathbb{C}):~\sup_{\mathbf{x}\in{\mathbb{R}}^{n}}\left|\mathbf{x}^{\mathbf{\alpha}}(\mathrm{D}^{\mathbf{\beta}}f)(\mathbf{x})\right|<\infty, ~\forall \mathbf{\alpha},\mathbf{\beta}\in\mathbb{N}^{n}\}\\
&=\{f\in\mathcal{C}^{\infty}(\mathbb{R}^{n},\mathbb{C}):~\nu_{N,k}(f):=\max_{|\beta|\leq k}\sup_{\mathbf{x}\in{\mathbb{R}}^{n}}(1+\|\mathbf{x}\|)^{N}\cdot |\mathrm{D}^{\beta}f(\mathrm{x})|<\infty, ~\forall N,k\in\mathbb{N}\},
\end{align*}
where for $\mathbf{x}=(x_1,\dots,x_n)\in\mathbb{R}^{n}$, $\alpha=(\alpha_1,\dots,\alpha_n),\beta=(\beta_1,\dots,\beta_n)\in\mathbb{N}^{n}$, we used the classical notations $\mathbf{x}^{\mathbf{\alpha}}:=x_1 ^{\alpha_1}\dots x_n ^{\alpha_n}$, $\|\mathbf{x}\|=(x_1 ^2+\dots+x_n ^2)^{1/2}$, and $\mathrm{D}^{\mathbf{\beta}}:=\partial_{1}^{\beta_1}\dots\partial_{n}^{\beta_n}$.

The Schwartz space, $\mathcal{S}(\mathbb{R}^{n})$, is a Fr\'echet vector space with the locally convex topology generated by the family of norms $\{\nu_{N,k}\}_{N,k\in\mathbb{N}}$. Let us recall that the space $\mathcal{C}^{\infty}_{c}(\mathbb{R}^{n})$, of compactly supported smooth functions on $\mathbb{R}^{n}$, is a dense subset of $\mathcal{S}(\mathbb{R}^{n})$. On the other hand, $\mathcal{S}(\mathbb{R}^{n})$ is dense in $L^{p}(\mathbb{R}^{n})$, for $1\leq p<\infty$.

As already mentioned above, naturally associated to the Schwartz space is the classical Fourier transform, $\mathcal{F}:\mathcal{S}(\mathbb{R}^{n})\rightarrow \mathcal{S}(\mathbb{R}^{n})$, defined for any $f\in\mathcal{S}(\mathbb{R}^{n})$ by the relation
\begin{equation*}
(\mathcal{F}f)(\xi):=\int_{\mathbb{R}^{n}}e^{-2\pi i\langle \mathbf{x},\mathbf{\xi}\rangle}f(\mathbf{x})\mathrm{d}\mathbf{x}, ~\forall\xi\in\mathbb{R}^{n}.
\end{equation*}
Recall that $\mathcal{F}$ is a continuous linear isomorphism of $\mathcal{S}(\mathbb{R}^{n})$.

Let us notice that the definition of the classical Fourier transform is given in terms of the canonical Euclidean inner product on $\mathbb{R}^{n}$, namely $\langle \cdot,\cdot \rangle$. This geometric feature of the Fourier transform implies a natural generalization by choosing a general geometric structure on $\mathbb{R}^{n}$ instead of the canonical Euclidean inner product. In contrast with the canonical Euclidean inner product which is symmetric, a general geometric structure needs not have symmetry--like properties, thus in the general case there are two natural choices to define a geometric Fourier transform, which will be given in the following definition.

\begin{definition}\label{defFT}
Let $b$ be a geometric structure on $\mathbb{R}^n$ and $f\in\mathcal{S}(\mathbb{R}^{n})$. The mappings
\begin{equation*}
(\mathcal{F}^{L}_{b} f)(\xi):=\int_{\mathbb{R}^{n}}e^{-2\pi i b(\xi,\mathbf{x})}f(\mathbf{x})\mathrm{d}\mathbf{x}, ~\forall\xi\in\mathbb{R}^{n},
\end{equation*}
\begin{equation*}
(\mathcal{F}^{R}_{b} f)(\xi):=\int_{\mathbb{R}^{n}}e^{-2\pi i b(\mathbf{x},\xi)}f(\mathbf{x})\mathrm{d}\mathbf{x}, ~\forall\xi\in\mathbb{R}^{n},
\end{equation*}
are called the \textit{left/right Fourier transform of $f$, associated to the geometric structure $b$}. 
\end{definition}
Note that for any arbitrary fixed $\xi\in\mathbb{R}^{n}$, each of the functions $$e^{-2\pi i\langle\cdot,\xi\rangle}, e^{-2\pi i b(\xi,\cdot)}, e^{-2\pi i b(\cdot,\xi)} :\mathbb{R}^{n}\rightarrow \mathbb{S}^{1},$$ is a character of the additive group $(\mathbb{R}^{n},+)$, i.e., a group homomorphism from $(\mathbb{R}^{n},+)$ to $(\mathbb{S}^{1},\cdot)$, where $\mathbb{S}^{1}=\{z\in\mathbb{C}: ~|z|=1\}$. Thus, from the point of view of abstract Fourier analysis, there are no essential differences between the classical Fourier transform and the left/right Fourier transforms. Nevertheless, the specific differences we shall focus on, are related to the compatibility between these transforms and the group of automorphisms of $\mathbb{R}^{n}$ which preserve the associated geometric structure, i.e., the orthogonal group, $O_{n}(\mathbb{R})$, associated to the canonical Euclidean inner product $\langle\cdot,\cdot\rangle$, in the case of the classical Fourier transform, and the group $G_b$, associated to the general geometric structure $b$, in the case of the left/right Fourier transforms.

Before stating next result, let us point out that to any geometric structure $b$ on $\mathbb{R}^{n}$ one can naturally associate another geometric structure, called the \textit{opposite geometric structure}, and given by $b^{op}:\mathbb{R}^{n}\times\mathbb{R}^{n}\rightarrow \mathbb{R}$, $b^{op}(\mathbf{x},\mathbf{y}):=b(\mathbf{y},\mathbf{x}), ~\forall\mathbf{x},\mathbf{y}\in\mathbb{R}^{n}$. In terms of geometric pairs, if $(b,B)$ is the geometric pair generated by $b$, then the geometric pair generated by $b^{op}$ is $(b^{op},B^{\top})$. A direct consequence of this fact is $\det b^{op}=\det b$. Note that a geometric structure $b$ is symmetric if and only if $b^{op}=b$, and skew--symmetric if and only if $n\in 2\mathbb{N}$ and $b^{op}= - b$.

Using this terminology, we point out some natural relations between left and right Fourier transforms associated to a given geometric structure, which follow directly from the Definition \ref{defFT}.
\begin{remark}\label{fop}
\begin{itemize}
\item[(i)] If $b$ is a geometric structure on $\mathbb{R}^{n}$ and $b^{op}$ the associated opposite geometric structure, then $\mathcal{F}^{L}_{b}=\mathcal{F}^{R}_{b^{op}}$ and $\mathcal{F}^{R}_{b}=\mathcal{F}^{L}_{b^{op}}$.
\item[(ii)] If $b$ is a symmetric geometric structure then $\mathcal{F}^{L}_{b}=\mathcal{F}^{R}_{b}$. Particularly, when $b$ is the canonical Euclidean inner product on $\mathbb{R}^{n}$ then $\mathcal{F}^{L}_{b}=\mathcal{F}^{R}_{b}=\mathcal{F}$, where $\mathcal{F}$ is the classical Fourier transform. 
\item[(iii)] If $n\in 2\mathbb{N}$ and $b$ is a skew--symmetric geometric structure, then $\mathcal{F}^{L}_{b}=\mathcal{F}^{R}_{-b}$,  or equivalently, $\mathcal{F}^{R}_{b}=\mathcal{F}^{L}_{-b}$.
\end{itemize}
\end{remark}
Next we provide a correspondence between left/right Fourier transforms and the classical Fourier transform. In order to do this we shall need the notion of geometric pair $(b,B)$ defined by the relation \eqref{relb}.

\begin{proposition}\label{proplr}
Let $b$ be a geometric structure on $\mathbb{R}^n$, $(b,B)$ the associated geometric pair, and $f\in\mathcal{S}(\mathbb{R}^{n})$. Then the following relations hold true
\begin{align}\label{left}
(\mathcal{F}^{L}_{b}f)(\xi)=(\mathcal{F}f)(B^{-\top}\xi), ~\forall \xi\in\mathbb{R}^{n},
\end{align}
\begin{align}\label{right}
(\mathcal{F}^{R}_{b}f)(\xi)=(\mathcal{F}f)(B^{-1}\xi), ~\forall \xi\in\mathbb{R}^{n},
\end{align}
where $B^{-\top}:=(B^{-1})^{\top}$.
\end{proposition}
\begin{proof}
Using the definition of the left Fourier transform and the equality \eqref{relb}, we obtain for any $\xi\in\mathbb{R}^{n}$
\begin{align*}
(\mathcal{F}^{L}_{b}f)(\xi) &=\int_{\mathbb{R}^{n}}e^{-2\pi i b(\xi,\mathbf{x})}f(\mathbf{x})\mathrm{d}\mathbf{x}=\int_{\mathbb{R}^{n}}e^{-2\pi i \langle\xi,B^{-1}\mathbf{x}\rangle}f(\mathbf{x})\mathrm{d}\mathbf{x}=\int_{\mathbb{R}^{n}}e^{-2\pi i \langle B^{-\top}\xi,\mathbf{x}\rangle}f(\mathbf{x})\mathrm{d}\mathbf{x}\\
&=\int_{\mathbb{R}^{n}}e^{-2\pi i \langle \mathbf{x}, B^{-\top}\xi\rangle}f(\mathbf{x})\mathrm{d}\mathbf{x}=(\mathcal{F}f)(B^{-\top}\xi),
\end{align*}
and hence we proved the relation \eqref{left}.

Similarly, from definition of the right Fourier transform and the equality \eqref{relb}, we obtain for any $\xi\in\mathbb{R}^{n}$
\begin{align*}
(\mathcal{F}^{R}_{b}f)(\xi) &=\int_{\mathbb{R}^{n}}e^{-2\pi i b(\mathbf{x},\xi)}f(\mathbf{x})\mathrm{d}\mathbf{x}=\int_{\mathbb{R}^{n}}e^{-2\pi i \langle \mathbf{x},B^{-1}\xi \rangle}f(\mathbf{x})\mathrm{d}\mathbf{x}=(\mathcal{F}f)(B^{-1}\xi),
\end{align*}
thus we get the relation \eqref{right}.
\end{proof}

The identities given in Proposition \ref{proplr} can be compactly written using the action of the automorphisms group $\operatorname{Aut}(\mathbb{R}^{n})$ on the Schwartz space $\mathcal{S}(\mathbb{R}^{n})$, defined by
\begin{equation}\label{actiune}
(\tau_{A} f)(\mathbf{x}):=f(A^{-1}\mathbf{x}), \forall\mathbf{x}\in\mathbb{R}^{n},
\end{equation}
for any $A\in\operatorname{Aut}(\mathbb{R}^{n})$ and $f\in\mathcal{S}(\mathbb{R}^{n})$. Note that for any fixed $A\in\operatorname{Aut}(\mathbb{R}^{n})$, the mapping $\tau_A :\mathcal{S}(\mathbb{R}^{n})\rightarrow \mathcal{S}(\mathbb{R}^{n})$ is a continuous linear isomorphism of $\mathcal{S}(\mathbb{R}^{n})$.

Before stating next result, let us recall that the classical Fourier transform $\mathcal{F}:\mathcal{S}(\mathbb{R}^{n})\rightarrow\mathcal{S}(\mathbb{R}^{n})$ is a continuous linear isomorphism whose inverse is given by
\begin{equation}\label{invF}
\mathcal{F}^{-1}=\tau_{-I_{d}}\circ\mathcal{F}=\mathcal{F}\circ\tau_{-I_{d}}.
\end{equation}
Moreover, the following classical relation holds
\begin{equation}\label{rimpo2}
\mathcal{F}\circ\tau_{A^{-1}}=\dfrac{1}{|\det A|}\cdot (\tau_{A^{\top}}\circ\mathcal{F}), ~\forall A\in\operatorname{Aut}(\mathbb{R}^{n}).
\end{equation}
Let us provide now some natural relations between the left/right Fourier transforms associated to a general geometric structure and the classical Fourier transform.
\begin{theorem}\label{teo2}
Let $b$ be a geometric structure on $\mathbb{R}^n$ and $(b,B)$ the associated geometric pair. Then the following assertions hold
\begin{itemize}
\item[(i)] $\mathcal{F}^{L}_{b}=\tau_{B^{\top}}\circ\mathcal{F}$,
\item[(ii)] $\mathcal{F}^{R}_{b}=\tau_{B}\circ\mathcal{F}$,
\item[(iii)] $\mathcal{F}^{R}_{b}=\tau_{BB^{-\top}}\circ\mathcal{F}^{L}_{b}$,
\item[(iv)] $\mathcal{F}^{L}_{b},\mathcal{F}^{R}_{b}:\mathcal{S}(\mathbb{R}^{n})\rightarrow \mathcal{S}(\mathbb{R}^{n})$ are continuous linear isomorphisms of the Schwartz space,
\item[(v)] $(\mathcal{F}^{L}_{b})^{-1}=\mathcal{F}\circ\tau_{-B^{-\top}}=|\det b|\cdot(\tau_{-B}\circ\mathcal{F})=\tau_{B^{-\top}}\circ\mathcal{F}^{L}_{b} \circ\tau_{-B^{-\top}}=|\det b|\cdot (\tau_{-I_{d}}\circ\mathcal{F}^{R}_{b})$,
\item[(vi)] $(\mathcal{F}^{R}_{b})^{-1}=\mathcal{F}\circ\tau_{-B^{-1}}=|\det b|\cdot(\tau_{-B^{\top}}\circ\mathcal{F})=\tau_{B^{-1}}\circ\mathcal{F}^{R}_{b} \circ\tau_{-B^{-1}}=|\det b|\cdot (\tau_{-I_{d}}\circ\mathcal{F}^{L}_{b})$,
\item[(vii)] $\mathcal{F}^{L/R}_{-b}=\tau_{-I_{d}}\circ\mathcal{F}^{L/R}_{b}=\dfrac{1}{|\det b|}\cdot (\mathcal{F}^{R/L}_{b})^{-1}$.
\end{itemize}
\end{theorem}
\begin{proof}
\begin{itemize}
\item[(i),(ii)] Both identities follow directly from the relations \eqref{left}, \eqref{right} and \eqref{actiune}.
\item[(iii)] Using the equalities $(i)$ and $(ii)$ we obtain
\begin{align*}
\tau_{B^{-1}} \circ \mathcal{F}^{R}_{b}=\tau_{B^{-\top}} \circ \mathcal{F}^{L}_{b} \Leftrightarrow \mathcal{F}^{R}_{b}=\tau_{B} \circ \tau_{B^{-\top}} \circ \mathcal{F}^{L}_{b} \Leftrightarrow \mathcal{F}^{R}_{b}=\tau_{BB^{-\top}}\circ\mathcal{F}^{L}_{b},
\end{align*}
and thus we get the conclusion.
\item[(iv)] From $(i)$ and $(ii)$ it follows that the left/right Fourier transforms are continuous linear isomorphisms of the Schwartz space as compositions of continuous linear isomorphisms of the Schwartz space. 
\item[(v)] As $\mathcal{F}^{L}_{b}$ is a linear isomorphism of the Schwartz space, the relations \eqref{invF} together with item $(i)$ imply
\begin{align*}
(\mathcal{F}^{L}_{b})^{-1}&=(\tau_{B^{\top}}\circ\mathcal{F})^{-1}=\mathcal{F}^{-1}\circ \tau_{B^{\top}}^{-1}=(\mathcal{F}\circ\tau_{-I_{d}})\circ\tau_{B^{-\top}}=\\
&=\mathcal{F}\circ(\tau_{-I_{d}}\circ\tau_{B^{-\top}})=\mathcal{F}\circ\tau_{-B^{-\top}},
\end{align*}
which yields the first equality of $(v)$. 

Second equality, $\mathcal{F}\circ\tau_{-B^{-\top}}=|\det b|\cdot(\tau_{-B}\circ\mathcal{F})$, is a direct consequence of \eqref{rimpo2} taking into account that $\det b =(\det B)^{-1}$.

Next we prove that $(\mathcal{F}^{L}_{b})^{-1}=\tau_{B^{-\top}}\circ\mathcal{F}^{L}_{b} \circ\tau_{-B^{-\top}}$. This follows from the first equality of $(v)$ together with item $(i)$, namely
\begin{align*}
(\mathcal{F}^{L}_{b})^{-1}&=\mathcal{F}\circ\tau_{-B^{-\top}}=(\tau_{B^{-\top}}\circ\tau_{B^{\top}})\circ\mathcal{F}\circ\tau_{-B^{-\top}}=\tau_{B^{-\top}}\circ (\tau_{B^{\top}}\circ\mathcal{F})\circ\tau_{-B^{-\top}}\\
&=\tau_{B^{-\top}}\circ\mathcal{F}_{b}^{L}\circ\tau_{-B^{-\top}}.
\end{align*}

Finally we show that $|\det b|\cdot(\tau_{-B}\circ\mathcal{F})=|\det b|\cdot (\tau_{-I_{d}}\circ\mathcal{F}^{R}_{b})$. This relation is obviously equivalent to $\tau_{-B}\circ\mathcal{F}=\tau_{-I_{d}}\circ\mathcal{F}^{R}_{b}$ which is in turn equivalent to $(ii)$.
\item[(vi)] The proof of this item is similar to that of the item $(v)$.
\item[(vii)] Using the relation $(i)$ for the geometric pair $(-b,-B)$ and then for the geometric pair $(b,B)$, we obtain successively 
\begin{align*}
\mathcal{F}^{L}_{-b}&=\tau_{-B^{\top}}\circ\mathcal{F}=\tau_{-I_{d}}\circ(\tau_{B^{\top}}\circ\mathcal{F})\\
&=\tau_{-I_{d}}\circ\mathcal{F}^{L}_{b},
\end{align*}
hence we obtained the first equality concerning the left Fourier transform. A similar argument can be used in the case of the right Fourier transform, this time using the relation $(ii)$ instead of $(i)$.

Let us now prove the second identity from $(vii)$. From items $(v)$ and $(vi)$ we have that $(\mathcal{F}^{R/L}_{b})^{-1}=|\det b|\cdot (\tau_{-I_{d}}\circ\mathcal{F}^{L/R}_{b})$. Since  this relation is obviously equivalent to $\tau_{-I_{d}}\circ\mathcal{F}^{L/R}_{b}=\dfrac{1}{|\det b|}\cdot (\mathcal{F}^{R/L}_{b})^{-1}$, we get the conclusion.
\end{itemize}
\end{proof}

Let us provide now some explicit formulas for the inverse left/right Fourier transforms associated to a general geometric structure, which follow directly from Theorem \ref{teo2} $(vii)$. 
\begin{remark}\label{invFTT}
Let $b$ be a geometric structure on $\mathbb{R}^n$ and $f\in\mathcal{S}(\mathbb{R}^{n})$. Then writing explicitly the relations, $(\mathcal{F}^{L/R}_{b})^{-1}f=|\det b|\cdot \mathcal{F}^{R/L}_{-b}f$, one obtain
\begin{equation*}
[(\mathcal{F}^{L}_{b})^{-1} f](\mathbf{x})=|\det b|\cdot\int_{\mathbb{R}^{n}}e^{2\pi i b(\xi,\mathbf{x})}f(\xi)\mathrm{d}\xi, ~\forall \mathbf{x}\in\mathbb{R}^{n},
\end{equation*}
\begin{equation*}
[(\mathcal{F}^{R}_{b})^{-1} f](\mathbf{x})=|\det b|\cdot\int_{\mathbb{R}^{n}}e^{2\pi i b(\mathbf{x},\xi)}f(\xi)\mathrm{d}\xi, ~\forall \mathbf{x}\in\mathbb{R}^{n}.
\end{equation*}
\end{remark}
Next result proves that implications $(ii), (iii)$ from Remark \ref{fop} are actually equivalences, thus providing characterizations of symmetric and skew--symmetric geometric structures in terms of the associated left/right Fourier transforms.  

\begin{theorem}
Let $b$ be a geometric structure on $\mathbb{R}^n$. Then 
\begin{itemize}
\item[(i)] $\mathcal{F}^{L}_{b}=\mathcal{F}^{R}_{b}$ if and only if $b$ is symmetric, i.e., $b$ is an Euclidean or a pseudo--Euclidean inner product,
\item[(ii)] $\mathcal{F}^{L/R}_{b}=\dfrac{1}{|\det b|}\cdot(\mathcal{F}^{L/R}_{b})^{-1}$ if and only if $n\in 2 \mathbb{N}$ and $b$ is skew--symmetric, i.e., $b$ is a symplectic structure,
\item[(iii)] $\mathcal{F}^{L/R}_{b}=\mathcal{F}^{R/L}_{-b}$ if and only if $n\in 2 \mathbb{N}$ and $b$ is skew--symmetric.
\end{itemize}
\end{theorem}
\begin{proof}
\begin{itemize}
\item[(i)] First equivalence is a direct consequence of Theorem \ref{teo2} $(iii)$. Indeed, the following sequence of equivalences hold 
\begin{align*}
\mathcal{F}^{L}_{b}=\mathcal{F}^{R}_{b} &\Leftrightarrow \mathcal{F}^{L}_{b}=\tau_{BB^{-\top}}\circ\mathcal{F}^{L}_{b}\Leftrightarrow \operatorname{Id}=\tau_{BB^{-\top}} \Leftrightarrow B^{\top}=B\\
& \Leftrightarrow b^{op}=b \Leftrightarrow b ~\text{is symmetric}.
\end{align*}
\item[(ii)] We shall prove the result only for the left Fourier transform, the case involving the right Fourier transform being similar. The equivalence follows immediately from Theorem \ref{teo2} $(iii),(vii)$.  More precisely, we have
\begin{align*}
&\mathcal{F}^{L}_{b}=\dfrac{1}{|\det b|}\cdot(\mathcal{F}^{L}_{b})^{-1} \Leftrightarrow \mathcal{F}^{L}_{b}=\mathcal{F}^{R}_{-b} \Leftrightarrow \mathcal{F}^{L}_{b}=\tau_{-I_{d}}\circ\mathcal{F}^{R}_{b} \\
&\Leftrightarrow \mathcal{F}^{L}_{b}=\tau_{-I_{d}}\circ \tau_{BB^{-\top}}\circ\mathcal{F}^{L}_{b} \Leftrightarrow  \mathcal{F}^{L}_{b}=\tau_{-BB^{-\top}}\circ\mathcal{F}^{L}_{b}\\
& \Leftrightarrow \operatorname{Id}=\tau_{-BB^{-\top}} \Leftrightarrow n\in 2\mathbb{N} ~\text{and} ~B^{\top}=-B \Leftrightarrow n\in 2\mathbb{N} ~\text{and} ~ b^{op}=-b \\
&\Leftrightarrow n\in 2\mathbb{N} ~\text{and} ~ b ~\text{is skew--symmetric}.
\end{align*}
\item[(iii)] The proof of this item follows directly from $(ii)$ and Theorem \ref{teo2} $(vii)$. 
\end{itemize}
\end{proof}

Now we give a result which reveals an important geometric property of the left/right Fourier transforms associated to a general geometric structure $b$ on $\mathbb{R}^{n}$. Before stating the result let us recall that to any given geometric structure $b$, we associate the subgroup of automorphisms which are compatible with $b$, denoted by $G_b \leq \operatorname{Aut}{\mathbb{R}^{n}}$, and defined as follows
$$
G_b =\{A\in\operatorname{End}{\mathbb{R}^{n}}:~b(A\mathbf{x},A\mathbf{y})=b(\mathbf{x},\mathbf{y}),~\forall\mathbf{x},\mathbf{y}\in\mathbb{R}^{n}\}.
$$ 
Next result shows that to any arbitrary fixed geometric structure $b$, the left/right Fourier transforms $\mathcal{F}^{L/R}_{b}$ as well as their inverses, are all compatible with the subgroup of automorphisms which preserve $b$. 
\begin{theorem}\label{ginvFT}
Let $b$ be a geometric structure on $\mathbb{R}^n$. Then for any $A\in G_b$
\begin{equation}\label{invFR}
(\mathcal{F}^{L/R}_{b})^{\pm 1}\circ\tau_{A}=\tau_{A}\circ (\mathcal{F}^{L/R}_{b})^{\pm 1}.
\end{equation}
\end{theorem}
\begin{proof}
First of all let us recall from \eqref{ecimp2} that in terms of the geometric pair $(b,B)$ associated to the geometric structure $b$, the group $G_b$ can be equivalently written as $G_b =\{A\in\operatorname{End}{\mathbb{R}^{n}}:~ABA^{\top}=B\}$. Thus for any $A\in G_b$ we have $|\det A|=1$. Consequently, as $G_{b} \leq \operatorname{Aut}(\mathbb{R}^{n})$, the equality \eqref{rimpo2} implies that 
\begin{equation}\label{superel}
\mathcal{F}\circ\tau_{A}=\tau_{A^{-\top}}\circ\mathcal{F}, ~\forall A\in G_b.
\end{equation}
Let us show now that for any $A\in G_b$
\begin{align}\label{yty}
\mathcal{F}^{L/R}_{b}\circ\tau_{A}=\tau_{A}\circ\mathcal{F}^{L/R}_{b}.
\end{align}
We shall prove the above identity only for the left Fourier transform, the relation for the right Fourier transform following similarly. Indeed, from Theorem \ref{teo2} $(i)$ and the relation \eqref{superel}, for any $A\in G_b$ we have
\begin{align*}
\mathcal{F}^{L}_{b}\circ\tau_{A}&=(\tau_{B^{\top}}\circ\mathcal{F})\circ\tau_{A}=\tau_{B^{\top}}\circ(\mathcal{F}\circ\tau_{A})=\tau_{B^{\top}}\circ(\tau_{A^{-\top}}\circ\mathcal{F})=(\tau_{B^{\top}}\circ\tau_{A^{-\top}})\circ\mathcal{F}\\
&=\tau_{B^{\top}A^{-\top}}\circ\mathcal{F}=\tau_{AB^{\top}}\circ\mathcal{F}=(\tau_{A}\circ\tau_{B^{\top}})\circ\mathcal{F}=\tau_{A}\circ(\tau_{B^{\top}}\circ\mathcal{F})\\
&=\tau_{A}\circ\mathcal{F}^{L}_{b},
\end{align*}
where we used the fact that $A\in G_{b}$ if and only if $B^{\top}A^{-\top}=AB^{\top}$.

In order to prove that for any $A\in G_b$
\begin{align}\label{uiy}
(\mathcal{F}^{L/R}_{b})^{-1}\circ\tau_{A}=\tau_{A}\circ(\mathcal{F}^{L/R}_{b})^{-1},
\end{align}
the relation $(\mathcal{F}^{L/R}_{b})^{-1}=|\det b|\cdot\mathcal{F}^{R/L}_{-b}$ (cf. Theorem \ref{teo2} $(vii)$), and the $\mathbb{R}-$linearity of $\tau_{A}$ imply that $\eqref{uiy}$ is equivalent to 
\begin{align}\label{rty}
\mathcal{F}^{R/L}_{-b}\circ\tau_{A}=\tau_{A}\circ \mathcal{F}^{R/L}_{-b},~\forall A\in G_{b}.
\end{align}
Since $G_{b}=G_{-b}$, the relation \eqref{rty} coincides with the identity \eqref{yty} for the geometric structure $-b$, and thus we get the conclusion. 
\end{proof}

A direct consequence of Theorem \ref{ginvFT} is the following result concerning $G_{b}-$invariant Schwartz functions. Before stating the result, let us recall that a general scalar function $f:\mathbb{R}^{n}\rightarrow \mathbb{R}$ is $G_{b}-$invariant if and only if $\tau_{A} f=f, ~\forall A\in G_{b}$. 
\begin{corollary}
Let $b$ be a geometric structure on $\mathbb{R}^n$. Then for any $G_{b}-$invariant function $f\in\mathcal{S}(\mathbb{R}^{n})$, its left/right Fourier transforms, $\mathcal{F}^{L/R}_{b}f$, as well as the inverse left/right Fourier transforms, $(\mathcal{F}^{L/R}_{b})^{-1}f$, are all of them $G_{b}-$invariant Schwartz functions. 
\end{corollary}
In the following we present a result concerning the compatibility between the left/right Fourier transforms of a Schwartz function and those of the corresponding complex conjugate function.
\begin{proposition}\label{complex}
Let $b$ be a geometric structure on $\mathbb{R}^n$ and $f\in\mathcal{S}(\mathbb{R}^{n})$. Then $$\overline{\mathcal{F}^{L/R}_{b}f}=(\tau_{-I_{d}}\circ\mathcal{F}^{L/R}_{b})(\overline{f})=\mathcal{F}^{L/R}_{-b} \overline{f},$$ where for any $g\in\mathcal{S}(\mathbb{R}^{n})\subset\mathcal{C}^{\infty}(\mathbb{R}^{n},\mathbb{C})$, $\overline{g}(\mathbf{x}):=\overline{g(\mathbf{x})}, ~\forall\mathbf{x}\in\mathbb{R}^{n}$.
\end{proposition}
\begin{proof}
We shall prove only the first identity, since the second one follows directly from Theorem \ref{teo2} $(vii)$. Let us start by recalling a standard identity for the classical Fourier transform, namely
\begin{equation}\label{conjF}
\overline{\mathcal{F}f}=(\tau_{-I_{d}}\circ\mathcal{F})(\overline{f}), ~\forall f\in\mathcal{S}(\mathbb{R}^{n}).
\end{equation}
Using the above result, we will show that for any $f\in\mathcal{S}(\mathbb{R}^{n})$, 
\begin{align}\label{LconjF}
\overline{\mathcal{F}^{L}_{b}f}=(\tau_{-I_{d}}\circ\mathcal{F}^{L}_{b})(\overline{f}).
\end{align}
In order to do that, let $(b,B)$ be the geometric pair associated to the geometric structure $b$, and let $f\in\mathcal{S}(\mathbb{R}^{n})$ be arbitrary fixed. Now using the relation \eqref{conjF}, the following equalities hold for any $\xi\in\mathbb{R}^{n}$
\begin{align*}
\overline{(\mathcal{F}^{L}_{b}f)}(\xi)&=\overline{(\mathcal{F}^{L}_{b}f)(\xi)}=\overline{\mathcal{F}f(B^{-\top}\xi)}=\mathcal{F}\overline{f}(-B^{-\top}\xi)=\mathcal{F}\overline{f}(B^{-\top}\cdot(-\xi))\\
&=(\mathcal{F}^{L}_{b}\overline{f})(-\xi)=[(\tau_{-I_{d}}\circ\mathcal{F}^{L}_{b})(\overline{f})](\xi),
\end{align*}
leading to the identity \eqref{LconjF}.
The similar relation concerning the right Fourier transform associated to the geometric structure $b$, follows mimetically. Note that both identities can be also proved analytically by direct computations using just the Definition \ref{defFT}.
\end{proof}

In order to state the following result, we need to introduce first some notations. Specifically, for any $f,g\in\mathcal{S}(\mathbb{R}^{n})$ we denote
\begin{equation*}
\langle f,g \rangle:=\int_{\mathbb{R}^{n}}f(\mathbf{x})g(\mathbf{x})\mathrm{d}\mathbf{x},
\end{equation*}
\begin{equation}\label{HS}
(f,g):=\langle f,\overline{g}\rangle=\int_{\mathbb{R}^{n}}f(\mathbf{x})\overline{g(\mathbf{x})}\mathrm{d}\mathbf{x},
\end{equation}
where \eqref{HS} is the usual Hermitian inner product on $\mathcal{S}(\mathbb{R}^{n})$. 

Next we provide an equivalent of Parseval's formula for the case of left/right Fourier transforms associated to a general geometric structure.
\begin{proposition}\label{L2}
Let $b$ be a geometric structure on $\mathbb{R}^n$. Then for any $f,g\in\mathcal{S}(\mathbb{R}^{n})$ we have
\begin{itemize}
\item[(i)] $\langle \mathcal{F}^{L/R}_{b}f,g\rangle=\langle f,\mathcal{F}^{R/L}_{b}g \rangle$,
\item[(ii)]$(f,g)=|\det b|\cdot (\mathcal{F}^{L/R}_{b}f,\mathcal{F}^{L/R}_{b}g)$,
\item[(iii)]$\int_{\mathbb{R}^{n}}|f(\mathbf{x})|^{2}\mathrm{d}\mathbf{x}=|\det b|\cdot\int_{\mathbb{R}^{n}}|(\mathcal{F}^{L/R}_{b}f)(\xi)|^{2}\mathrm{d}\xi$,
\item[(iv)] $\langle (\mathcal{F}^{L/R}_{b})^{-1}f,g\rangle=\langle f,(\mathcal{F}^{R/L}_{b})^{-1}g \rangle$,
\item[(v)]$(f,g)=|\det b|^{-1}\cdot ((\mathcal{F}^{L/R}_{b})^{-1}f,(\mathcal{F}^{L/R}_{b})^{-1}g)$,
\item[(vi)]$\int_{\mathbb{R}^{n}}|f(\mathbf{x})|^{2}\mathrm{d}\mathbf{x}=|\det b|^{-1}\cdot\int_{\mathbb{R}^{n}}|(\mathcal{F}^{L/R}_{b}f)^{-1}(\mathbf{x})|^{2}\mathrm{d}\mathbf{x}$.
\end{itemize}
\end{proposition}
\begin{proof}
\begin{itemize}
\item[(i)] Note that it is enough to prove the identity 
\begin{equation}\label{erw}
\langle \mathcal{F}^{L}_{b}f,g\rangle=\langle f,\mathcal{F}^{R}_{b}g \rangle,~\forall f,g\in\mathcal{S}(\mathbb{R}^{n}),
\end{equation}
since the similar relation, 
\begin{equation*}
\langle \mathcal{F}^{R}_{b}f,g\rangle=\langle f,\mathcal{F}^{L}_{b}g \rangle,~\forall f,g\in\mathcal{S}(\mathbb{R}^{n}),
\end{equation*}
follows by simply interchanging $f$ and $g$ in \eqref{erw} and using the commutativity of $\langle\cdot,\cdot\rangle$. 

In order to prove the identity \eqref{erw} we just use Fubini's theorem and Definition \ref{defFT}. Indeed, for any $f,g\in\mathcal{S}(\mathbb{R}^{n})$ we have
\begin{align*}
\langle \mathcal{F}^{L}_{b}f,g\rangle &=\int_{\mathbb{R}^{n}}(\mathcal{F}^{L}_{b}f)(\mathbf{y})\cdot g(\mathbf{y})\mathrm{d}\mathbf{y}=\int_{\mathbb{R}^{n}}\left(\int_{\mathbb{R}^{n}}e^{-2\pi i b(\mathbf{y},\mathbf{x})}f(\mathbf{x})g(\mathbf{y})\mathrm{d}\mathbf{x}\right)\mathrm{d}\mathbf{y}\\
&=\int_{\mathbb{R}^{n}}\left(\int_{\mathbb{R}^{n}}e^{-2\pi i b(\mathbf{x},\mathbf{y})}g(\mathbf{x})f(\mathbf{y})\mathrm{d}\mathbf{x}\right)\mathrm{d}\mathbf{y}=\int_{\mathbb{R}^{n}}(\mathcal{F}^{R}_{b}g)(\mathbf{y})\cdot f(\mathbf{y})\mathrm{d}\mathbf{y}\\
&=\int_{\mathbb{R}^{n}} f(\mathbf{y})\cdot(\mathcal{F}^{R}_{b}g)(\mathbf{y}) \mathrm{d}\mathbf{y}=\langle f,\mathcal{F}^{R}_{b}g \rangle,
\end{align*}
and thus we get the conclusion.
\item[(ii)] We shall prove the required identity only for the left Fourier transform, since the relation concerning the right Fourier transform follows mimetically. Now using the previous item, Proposition \ref{complex}, Theorem \ref{teo2} $(vii)$, and the linearity of the left/right Fourier transforms, we obtain for any $f,g\in\mathcal{S}(\mathbb{R}^{n})$ the following equalities
\begin{align*}
(\mathcal{F}^{L}_{b}f,\mathcal{F}^{L}_{b}g)&=\langle\mathcal{F}^{L}_{b}f,\overline{\mathcal{F}^{L}_{b}g}\rangle = \langle \mathcal{F}^{L}_{b}f,(\tau_{-I_{d}}\circ\mathcal{F}^{L}_{b})(\overline{g})\rangle = \langle f,(\mathcal{F}^{R}_{b}\circ\tau_{-I_{d}}\circ\mathcal{F}^{L}_{b}) (\overline{g})\rangle\\
&=\langle f,(\mathcal{F}^{R}_{b}\circ(|\det b|^{-1}\cdot (\mathcal{F}^{R}_{b})^{-1})) (\overline{g})\rangle = |\det b|^{-1}\cdot\langle f, \overline{g}\rangle \\
&= |\det b|^{-1}\cdot (f,g),
\end{align*}
which yields the desired identity.
\item[(iii)] This relation follows directly from item $(ii)$ for $g=f$.
\item[(iv),(v),(vi)] All these identities follow by applying previous items correspondingly for the geometric structure $-b$, then using the relation $\mathcal{F}^{R/L}_{-b}=\dfrac{1}{|\det b|}\cdot(\mathcal{F}^{L/R}_{b})^{-1}$ (cf. Theorem \ref{teo2} $(vii)$), and taking into account that $|\det(-b)|=|\det b|$.
\end{itemize}
\end{proof}

A consequence of the above proposition is the following result concerning the extension of left/right Fourier transforms to $L^{2}(\mathbb{R}^{n})$.
\begin{remark}
For any given geometric structure $b$ on $\mathbb{R}^{n}$, the associated left/right Fourier transforms extend by continuity from the dense subspace $\mathcal{S}(\mathbb{R}^{n})\subset L^{2}(\mathbb{R}^{n})$ to two isomorphisms, also denoted by $\mathcal{F}^{L/R}_{b}:L^{2}(\mathbb{R}^{n})\rightarrow L^{2}(\mathbb{R}^{n})$, such that
\begin{equation*}
\|\mathcal{F}^{L/R}_{b}f\|_{L^{2}(\mathbb{R}^{n})}=\dfrac{1}{\sqrt{|\det b|}}\cdot\|f\|_{L^{2}(\mathbb{R}^{n})}, ~\forall f\in L^{2}(\mathbb{R}^{n}).
\end{equation*}
\end{remark}
In the following we provide a generalization of the classical results which relates the Fourier transform of the convolution of two functions with the product of their Fourier transforms, and the Fourier transform of the product of two functions with the convolution of their Fourier transforms, respectively.
\begin{proposition}\label{staFT}
Let $b$ be a geometric structure on $\mathbb{R}^n$. Then for any $f,g\in\mathcal{S}(\mathbb{R}^{n})$ the following assertions hold
\begin{itemize}
\item[(i)] $\mathcal{F}^{L/R}_{b}(f\star g)=\mathcal{F}^{L/R}_{b}f \cdot \mathcal{F}^{L/R}_{b}g$,
\item[(ii)] $\mathcal{F}^{L/R}_{b}(f\cdot g)=|\det b|\cdot(\mathcal{F}^{L/R}_{b}f \star \mathcal{F}^{L/R}_{b}g)$.
\end{itemize}
\end{proposition}
\begin{proof} We shall prove the above identities only for the left Fourier transform, since the relations concerning the right Fourier transform follow mimetically.
\begin{itemize}
\item[(i)] We start by recalling a standard identity concerning the classical Fourier transform:
\begin{align}\label{starFT}
\mathcal{F}(f\star g)=\mathcal{F}f\cdot\mathcal{F}g, ~\forall f,g\in\mathcal{S}(\mathbb{R}^{n}).
\end{align}
Let $f,g\in\mathcal{S}(\mathbb{R}^{n})$ be arbitrary fixed. Using the Theorem \ref{teo2} $(i)$ and the identity \eqref{starFT} we have
\begin{align*}
\mathcal{F}^{L}_{b}(f\star g)&=(\tau_{B^{\top}}\circ\mathcal{F})(f\star g)=\tau_{B^{\top}}(\mathcal{F}(f\star g))=\tau_{B^{\top}}(\mathcal{F}f\cdot\mathcal{F}g)\\
&=(\tau_{B^{\top}}(\mathcal{F}f))\cdot(\tau_{B^{\top}}(\mathcal{F}g))=\mathcal{F}^{L}_{b}f \cdot \mathcal{F}^{L}_{b}g,
\end{align*} 
thus we get the desired identity.
\item[(ii)] Applying previous item for the geometric structure $-b$ and using the relation $\mathcal{F}^{R/L}_{-b}=\dfrac{1}{|\det b|}\cdot(\mathcal{F}^{L/R}_{b})^{-1}$ (cf. Theorem \ref{teo2} $(vii)$), we obtain
\begin{align*}
(\mathcal{F}^{L/R}_{b})^{-1}(F\star G)=\dfrac{1}{|\det b|}\cdot (\mathcal{F}^{L/R}_{b})^{-1}F \cdot (\mathcal{F}^{L/R}_{b})^{-1}G, ~\forall F,G\in\mathcal{S}(\mathbb{R}^{n}),
\end{align*}
or equivalently,
\begin{align}\label{ptr}
(\mathcal{F}^{L/R}_{b})^{-1}F \cdot (\mathcal{F}^{L/R}_{b})^{-1}G = |\det b|\cdot(\mathcal{F}^{L/R}_{b})^{-1}(F\star G) , ~\forall F,G\in\mathcal{S}(\mathbb{R}^{n}).
\end{align}
Since $\mathcal{F}^{L/R}_{b}$ are linear isomorphisms of the Schwartz space, for each $F,G\in\mathcal{S}(\mathbb{R}^{n})$ there exist and are unique $f,g\in \mathcal{S}(\mathbb{R}^{n})$ such that $F=\mathcal{F}^{L/R}_{b}f$, $G=\mathcal{F}^{L/R}_{b}g$. Thus the equality \eqref{ptr} becomes
\begin{align*}
f\cdot g = |\det b|\cdot(\mathcal{F}^{L/R}_{b})^{-1}( \mathcal{F}^{L/R}_{b}f \star \mathcal{F}^{L/R}_{b}g) , ~\forall f,g\in\mathcal{S}(\mathbb{R}^{n}),
\end{align*}
which due to the linearity of $\mathcal{F}^{L/R}_{b}$ is equivalent to the desired identity, i.e.,
\begin{align*}
\mathcal{F}^{L/R}_{b}(f\cdot g) = |\det b|\cdot (\mathcal{F}^{L/R}_{b}f \star \mathcal{F}^{L/R}_{b}g) , ~\forall f,g\in\mathcal{S}(\mathbb{R}^{n}).
\end{align*}
\end{itemize}
\end{proof}

A direct consequence of Proposition \ref{staFT} and of the identity $\mathcal{F}^{R/L}_{-b}=\dfrac{1}{|\det b|}\cdot(\mathcal{F}^{L/R}_{b})^{-1}$ (cf. Theorem \ref{teo2} $(vii)$) is the following result concerning the inverse left/right Fourier transforms associated to a geometric structure.
\begin{corollary}\label{staFTI}
Let $b$ be a geometric structure on $\mathbb{R}^n$. Then for any $f,g\in\mathcal{S}(\mathbb{R}^{n})$ the following assertions hold
\begin{itemize}
\item[(i)] $(\mathcal{F}^{L/R}_{b})^{-1}(f\star g)=|\det b|^{-1}\cdot (\mathcal{F}^{L/R}_{b})^{-1}f \cdot (\mathcal{F}^{L/R}_{b})^{-1}g$,
\item[(ii)] $(\mathcal{F}^{L/R}_{b})^{-1}(f\cdot g)=(\mathcal{F}^{L/R}_{b})^{-1}f \star (\mathcal{F}^{L/R}_{b})^{-1}g$.
\end{itemize}
\end{corollary}
Next we present a generalization of some classical properties of the standard Fourier transform to the case of left/right Fourier transforms associated to a general geometric structure.
\begin{proposition}\label{transF}
Let $b$ be a geometric structure on $\mathbb{R}^n$. Then for any $f\in\mathcal{S}(\mathbb{R}^{n})$, $\mathbf{h}\in\mathbb{R}^{n}$ and $\lambda >0$ the following assertions hold
\begin{itemize}
\item[(i)] $[\mathcal{F}^{L}_{b}f(\cdot+\mathbf{h})](\xi)=(\mathcal{F}^{L}_{b}f) (\xi) e^{2\pi i b(\xi,\mathbf{h})}, ~\forall \xi\in\mathbb{R}^{n}$,
\item[(ii)] $[\mathcal{F}^{R}_{b}f(\cdot+\mathbf{h})](\xi)=(\mathcal{F}^{R}_{b}f) (\xi) e^{2\pi i b(\mathbf{h},\xi)}, ~\forall \xi\in\mathbb{R}^{n}$,
\item[(iii)] $[\mathcal{F}^{L}_{b}(f(\cdot) e^{-2\pi i b(\mathbf{h},\cdot)})](\xi)=(\mathcal{F}^{L}_{b}f) (\xi+\mathbf{h}), ~\forall \xi\in\mathbb{R}^{n}$,
\item[(iv)] $[\mathcal{F}^{R}_{b}(f(\cdot) e^{-2\pi i b(\cdot,\mathbf{h})})](\xi)=(\mathcal{F}^{R}_{b}f) (\xi+\mathbf{h}), ~\forall \xi\in\mathbb{R}^{n}$,
\item[(v)] $[\mathcal{F}^{L/R}_{b}(\delta_{\lambda}f)](\xi)=\lambda^{-n}(\mathcal{F}^{L/R}f)(\lambda^{-1}\xi), ~\forall \xi\in\mathbb{R}^{n}$,\\
where $(\delta_{\lambda}f)(\mathbf{x}):=f(\lambda\mathbf{x}), ~\forall\mathbf{x}\in\mathbb{R}^{n}$.
\end{itemize}
\end{proposition}
\begin{proof}
\begin{itemize}
\item[(i)] Let us fix $f\in\mathcal{S}(\mathbb{R}^{n})$ and $\mathbf{h}\in\mathbb{R}^{n}$. Then for any $\xi\in\mathbb{R}^{n}$ we have
\begin{align*}
[\mathcal{F}^{L}_{b}f(\cdot+\mathbf{h})](\xi)&=\int_{\mathbb{R}^{n}}e^{-2\pi i b(\xi,\mathbf{x})}f(\mathbf{x}+\mathbf{h})\mathrm{d}\mathbf{x}=\int_{\mathbb{R}^{n}}e^{-2\pi i b(\xi,\mathbf{y}-\mathbf{h})}f(\mathbf{y})\mathrm{d}\mathbf{y}\\
&=e^{2\pi i b(\xi,\mathbf{h})} \int_{\mathbb{R}^{n}}e^{-2\pi i b(\xi,\mathbf{y})}f(\mathbf{y})\mathrm{d}\mathbf{y}=(\mathcal{F}^{L}_{b}f)(\xi) e^{2\pi i b(\xi,\mathbf{h})},
\end{align*}
and thus we get the conclusion.
\item[(ii)] The proof is similar to that of the item $(i)$.
\item[(iii)]  Let us fix $f\in\mathcal{S}(\mathbb{R}^{n})$ and $\mathbf{h}\in\mathbb{R}^{n}$. Then for any $\xi\in\mathbb{R}^{n}$ we have
\begin{align*}
[\mathcal{F}^{L}_{b}(f(\cdot) e^{-2\pi i b(\mathbf{h},\cdot)})](\xi)&=\int_{\mathbb{R}^{n}}e^{-2\pi i b(\xi,\mathbf{x})}f(\mathbf{x})e^{-2\pi i b(\mathbf{h},\mathbf{x})}\mathrm{d}\mathbf{x}=\int_{\mathbb{R}^{n}}e^{-2\pi i b(\xi+\mathbf{h},\mathbf{x})}f(\mathbf{x})\mathrm{d}\mathbf{x}\\
&=(\mathcal{F}^{L}_{b}f)(\xi+\mathbf{h}),
\end{align*}
and consequently we obtained the desired identity.
\item[(iv)] The proof is similar to that of the previous item.
\item[(v)] We shall prove only the identity concerning the left Fourier transform, the similar identity for the right Fourier transform following mimetically. In order to do that, let $f\in\mathcal{S}(\mathbb{R}^{n})$ and $\lambda>0$ be arbitrary fixed. Then for any $\xi\in\mathbb{R}^{n}$ we have
\begin{align*}
[\mathcal{F}^{L}_{b}(\delta_{\lambda}f)](\xi)&=\int_{\mathbb{R}^{n}}e^{-2\pi i b(\xi,\mathbf{x})}f(\lambda\mathbf{x})\mathrm{d}\mathbf{x}=\int_{\mathbb{R}^{n}}e^{-2\pi i b(\xi,\lambda^{-1}\mathbf{y})}f(\mathbf{y})\lambda^{-n}\mathrm{d}\mathbf{y}\\
&=\lambda^{-n}\int_{\mathbb{R}^{n}}e^{-2\pi i b(\lambda^{-1}\xi,\mathbf{y})}f(\mathbf{y})\mathrm{d}\mathbf{y}=\lambda^{-n}(\mathcal{F}^{L}_{b}f)(\lambda^{-1}\xi),
\end{align*}
hence we get the conclusion.
\end{itemize}
\end{proof}

Applying the above result for the geometric structure $-b$ and then using the identities $\mathcal{F}^{L/R}_{-b}=\dfrac{1}{|\det b|}\cdot(\mathcal{F}^{R/L}_{b})^{-1}$ (cf. Theorem \ref{teo2} $(vii)$) we obtain similar formulas for the inverse left/right Fourier transforms. More precisely, the following result holds.
\begin{corollary}
Let $b$ be a geometric structure on $\mathbb{R}^n$. Then for any $f\in\mathcal{S}(\mathbb{R}^{n})$, $\mathbf{h}\in\mathbb{R}^{n}$ and $\lambda >0$ we have
\begin{itemize}
\item[(i)] $[(\mathcal{F}^{R}_{b})^{-1}f(\cdot+\mathbf{h})](\mathbf{x})=[(\mathcal{F}^{R}_{b})^{-1}f] (\mathbf{x}) e^{-2\pi i b(\mathbf{x},\mathbf{h})}, ~\forall \mathbf{x}\in\mathbb{R}^{n}$,
\item[(ii)] $[(\mathcal{F}^{L}_{b})^{-1}f(\cdot+\mathbf{h})](\mathbf{x})=[(\mathcal{F}^{L}_{b})^{-1}f] (\mathbf{x}) e^{-2\pi i b(\mathbf{h},\mathbf{x})}, ~\forall \mathbf{x}\in\mathbb{R}^{n}$,
\item[(iii)] $[(\mathcal{F}^{R}_{b})^{-1}(f(\cdot) e^{2\pi i b(\mathbf{h},\cdot)})](\mathbf{x})=[(\mathcal{F}^{R}_{b})^{-1}f] (\mathbf{x}+\mathbf{h}), ~\forall \mathbf{x}\in\mathbb{R}^{n}$,
\item[(iv)] $[(\mathcal{F}^{L}_{b})^{-1}(f(\cdot) e^{2\pi i b(\cdot,\mathbf{h})})](\mathbf{x})=[(\mathcal{F}^{L}_{b})^{-1}f] (\mathbf{x}+\mathbf{h}), ~\forall \mathbf{x}\in\mathbb{R}^{n}$,
\item[(v)] $[(\mathcal{F}^{L/R}_{b})^{-1}(\delta_{\lambda}f)](\mathbf{x})=\lambda^{-n}[(\mathcal{F}^{L/R}_{b})^{-1}f](\lambda^{-1}\mathbf{x}), ~\forall \mathbf{x}\in\mathbb{R}^{n}$.
\end{itemize}
\end{corollary}
Let us now give a result which presents the compatibility between partial differentiation and the left/right Fourier transforms associated to a general geometric structure. This is a natural generalization of the similar result concerning the classical Fourier transform.
\begin{proposition}\label{simpor}
Let $b$ be a geometric structure on $\mathbb{R}^n$, $(b,B)$ the associated geometric pair, and $f\in\mathcal{S}(\mathbb{R}^{n})$. Then for any $\alpha\in\mathbb{N}$ and $j\in\{1,\dots,n\}$ we have that
\begin{itemize}
\item[(i)] $\mathcal{F}^{L}_{b}(\partial^{\alpha}_{j}f)=(2\pi i\cdot\tau_{B^{\top}}\pi_{j})^{\alpha}\cdot\mathcal{F}^{L}_{b}f$,
\item[(ii)] $\mathcal{F}^{R}_{b}(\partial^{\alpha}_{j}f)=(2\pi i\cdot\tau_{B}\pi_{j})^{\alpha}\cdot\mathcal{F}^{R}_{b}f$,
\item[(iii)] $(\mathcal{F}^{L}_{b})^{-1}(\partial^{\alpha}_{j}f)=(- 2\pi i\cdot\tau_{B}\pi_{j})^{\alpha}\cdot(\mathcal{F}^{L}_{b})^{-1}f$,
\item[(iv)] $(\mathcal{F}^{R}_{b})^{-1}(\partial^{\alpha}_{j}f)=(- 2\pi i\cdot\tau_{B^{\top}}\pi_{j})^{\alpha}\cdot(\mathcal{F}^{R}_{b})^{-1}f$,
\item[(v)] $\partial^{\alpha}_{j}\mathcal{F}^{L}_{b}f=(-2\pi i)^{\alpha}\cdot\mathcal{F}^{L}_{b}((\tau_{B}\pi_j)^{\alpha}\cdot f)$,
\item[(vi)] $\partial^{\alpha}_{j}\mathcal{F}^{R}_{b}f=(-2\pi i)^{\alpha}\cdot\mathcal{F}^{R}_{b}((\tau_{B^{\top}}\pi_j)^{\alpha}\cdot f)$,
\item[(vii)] $\partial^{\alpha}_{j}(\mathcal{F}^{L}_{b})^{-1}f=(2\pi i)^{\alpha}\cdot(\mathcal{F}^{L}_{b})^{-1}((\tau_{B^{\top}}\pi_j)^{\alpha}\cdot f)$,
\item[(viii)] $\partial^{\alpha}_{j}(\mathcal{F}^{R}_{b})^{-1}f=(2\pi i)^{\alpha}\cdot(\mathcal{F}^{R}_{b})^{-1}((\tau_{B}\pi_j)^{\alpha}\cdot f)$,
\end{itemize}
where $\pi_j :\mathbb{R}^{n}\rightarrow \mathbb{R}$, $\pi_j(x_1,\dots,x_n)=x_j, ~\forall (x_1,\dots,x_n)\in\mathbb{R}^{n}$, denotes the projection onto the $j$th factor.
\end{proposition}
\begin{proof}
First two items follow easily by using the similar result regarding the classical Fourier transform, which states that for any $f\in\mathcal{S}(\mathbb{R}^{n})$, $\alpha\in\mathbb{N}$ and $j\in\{1,\dots,n\}$ we have
\begin{align}\label{btr}
\mathcal{F}(\partial^{\alpha}_{j}f)=(2\pi i\cdot \pi_{j})^{\alpha}\cdot \mathcal{F}f,
\end{align}
where $\mathcal{F}$ denotes the classical Fourier transform.
\begin{itemize}
\item[(i)] Let us fix $f\in\mathcal{S}(\mathbb{R}^{n})$, $\alpha\in\mathbb{N}$ and $j\in\{1,\dots,n\}$. From Theorem \ref{teo2} $(i)$ and the relation \eqref{btr} we have 
\begin{align*}
\mathcal{F}^{L}_{b}(\partial^{\alpha}_{j}f)&=(\tau_{B^{\top}}\circ\mathcal{F})(\partial^{\alpha}_{j}f)=\tau_{B^{\top}}(\mathcal{F}(\partial^{\alpha}_{j}f))=\tau_{B^{\top}}((2\pi i\cdot \pi_{j})^{\alpha}\cdot \mathcal{F}f)\\
&=(2\pi i\cdot \tau_{B^{\top}}\pi_{j})^{\alpha}\cdot\tau_{B^{\top}}(\mathcal{F}f)=(2\pi i\cdot \tau_{B^{\top}}\pi_{j})^{\alpha}\cdot(\tau_{B^{\top}}\circ\mathcal{F})(f)\\
&=(2\pi i\cdot\tau_{B^{\top}}\pi_{j})^{\alpha}\cdot\mathcal{F}^{L}_{b}f,
\end{align*}
and hence we get the conclusion.
\item[(ii)] We shall proceed as in the proof of the previous item. In order to do that, let us fix $f\in\mathcal{S}(\mathbb{R}^{n})$, $\alpha\in\mathbb{N}$ and $j\in\{1,\dots,n\}$. From Theorem \ref{teo2} $(ii)$ and the relation \eqref{btr} we have 
\begin{align*}
\mathcal{F}^{R}_{b}(\partial^{\alpha}_{j}f)&=(\tau_{B}\circ\mathcal{F})(\partial^{\alpha}_{j}f)=\tau_{B}(\mathcal{F}(\partial^{\alpha}_{j}f))=\tau_{B}((2\pi i\cdot \pi_{j})^{\alpha}\cdot \mathcal{F}f)\\
&=(2\pi i\cdot \tau_{B}\pi_{j})^{\alpha}\cdot\tau_{B}(\mathcal{F}f)=(2\pi i\cdot \tau_{B}\pi_{j})^{\alpha}\cdot(\tau_{B}\circ\mathcal{F})(f)\\
&=(2\pi i\cdot\tau_{B}\pi_{j})^{\alpha}\cdot\mathcal{F}^{R}_{b}f,
\end{align*}
and thus we get the conclusion.
\item[(iii)] Applying the relation $(ii)$ for the geometric structure $-b$ and taking into account that $\tau_{-B}\pi_{j}=-\tau_{B}\pi_{j}$, we get for any $f\in\mathcal{S}(\mathbb{R}^{n})$, $\alpha\in\mathbb{N}$ and $j\in\{1,\dots,n\}$
\begin{align}\label{ppoo}
\mathcal{F}^{R}_{-b}(\partial^{\alpha}_{j}f)=(-2\pi i\cdot\tau_{B}\pi_{j})^{\alpha}\cdot\mathcal{F}^{R}_{-b}f.
\end{align}
Since $\mathcal{F}^{R}_{-b}=\dfrac{1}{|\det b|}\cdot(\mathcal{F}^{L}_{b})^{-1}$ (cf. Theorem \ref{teo2} $(vii)$), the relation \eqref{ppoo} leads to the required identity.
\item[(iv)] The proof is similar to that of the previous item. This time we apply the relation $(i)$ for the geometric structure $-b$, and then we use the identity $\mathcal{F}^{L}_{-b}=\dfrac{1}{|\det b|}\cdot(\mathcal{F}^{R}_{b})^{-1}$ (cf. Theorem \ref{teo2} $(vii)$).
\item[(v)] Note that it is enough to prove the identity for $\alpha=1$, i.e.,
\begin{align}\label{cazpartic}
\partial_{j}\mathcal{F}^{L}_{b}f= (-2\pi i)\cdot\mathcal{F}^{L}_{b}(\tau_{B}\pi_j \cdot f),
\end{align}
for any $f\in\mathcal{S}(\mathbb{R}^{n})$, and $j\in\{1,\dots,n\}$, since the general case follows trivially by induction on $\alpha$. In order to prove \eqref{cazpartic}, let us fix $f\in\mathcal{S}(\mathbb{R}^{n})$, and $j\in\{1,\dots,n\}$. Then the following equalities hold for any $\xi\in\mathbb{R}^{n}$
\begin{align*}
\dfrac{\partial}{\partial\xi_j}(\mathcal{F}^{L}_{b}f)(\xi)&=\dfrac{\partial}{\partial\xi_j}\int_{\mathbb{R}^{n}}e^{-2\pi i b(\xi,\mathbf{x})}~f(\mathbf{x})\mathrm{d}\mathbf{x}=\int_{\mathbb{R}^{n}}\dfrac{\partial}{\partial\xi_j}e^{-2\pi i b(\xi,\mathbf{x})}~f(\mathbf{x})\mathrm{d}\mathbf{x}\\
&=(-2\pi i)\cdot\int_{\mathbb{R}^{n}}e^{-2\pi i b(\xi,\mathbf{x})}~\dfrac{\partial}{\partial\xi_j}\langle\xi,B^{-1}\mathbf{x}\rangle ~f(\mathbf{x})\mathrm{d}\mathbf{x}\\
&=(-2\pi i)\cdot\int_{\mathbb{R}^{n}}e^{-2\pi i b(\xi,\mathbf{x})}~\pi_{j}{(B^{-1}\mathbf{x})} ~f(\mathbf{x})\mathrm{d}\mathbf{x}\\
&=(-2\pi i)\cdot\int_{\mathbb{R}^{n}}e^{-2\pi i b(\xi,\mathbf{x})}~(\tau_{B}\pi_{j})(\mathbf{x}) ~f(\mathbf{x})\mathrm{d}\mathbf{x}\\
&=(-2\pi i)\cdot[\mathcal{F}^{L}_{b}(\tau_{B}\pi_j \cdot f)](\xi),
\end{align*}
and thus we get the identity \eqref{cazpartic}. Note that the interchange of integration and differentiation is allowed, as $\sup_{\mathbf{x}\in\mathbb{R}^{n}}(1+\|\mathbf{x}\|^{2})^{n}~|\pi_{j}(B^{-1}\mathbf{x})\cdot f(\mathbf{x})|<\infty$ (since $f\in\mathcal{S}(\mathbb{R}^{n})$), and the integrand $e^{-2\pi i b(\xi,\mathbf{x})}~\pi_{j}{(B^{-1}\mathbf{x})} ~f(\mathbf{x})$ is continuous and dominated by the integrable function $$(1+\|\mathbf{x}\|^{2})^{-n}\cdot\sup_{\mathbf{x}\in\mathbb{R}^{n}}(1+\|\mathbf{x}\|^{2})^{n}~|\pi_{j}(B^{-1}\mathbf{x})\cdot f(\mathbf{x})|.$$
\item[(vi)] The proof is similar to that of the previous item.
\item[(vii)] The identity follows by applying $(vi)$ for the geometric structure $-b$, taking into account that $\tau_{-B^{\top}}\pi_j = -\tau_{B^{\top}}\pi_j$, and then using the linearity of the left/right Fourier transforms, and the formula $\mathcal{F}^{R}_{-b}=\dfrac{1}{|\det b|}\cdot(\mathcal{F}^{L}_{b})^{-1}$ (cf. Theorem \ref{teo2} $(vii)$).
\item[(viii)] The proof is similar to that of the item $(vii)$.
\end{itemize}
\end{proof}

In the following result we present a natural compatibility between the $b-$Laplace operator and hyperplanar waves, the main ingredients used in order to define the left/right Fourier transforms.
\begin{proposition}
Let $b$ be a geometric structure on $\mathbb{R}^n$. Then for any fixed $\xi\in\mathbb{R}^{n}$, the following relations hold 
\begin{itemize}
\item[(i)] $\Delta_{b}e^{\pm 2\pi i b(\mathbf{x},\xi)}=[-4\pi^{2} b(\xi,\xi)]\cdot e^{\pm 2\pi i b(\mathbf{x},\xi)}, ~\forall \mathbf{x}\in\mathbb{R}^{n}$,
\item[(ii)] $\Delta_{b}e^{\pm 2\pi i b(\mathbf{\xi},\mathbf{x})}=[-4\pi^{2} b(\xi,\xi)]\cdot e^{\pm 2\pi i b(\mathbf{\xi},\mathbf{x})}, ~\forall \mathbf{x}\in\mathbb{R}^{n}$.
\end{itemize}
\end{proposition}
\begin{proof} Let $(b,B)$ be the geometric pair associated to $b$ (see \eqref{relb}), and let $b_{ij}$, $i,j\in\{1,\dots,n\}$ be the entries of the matrix representing $B$ with respect to canonical basis. Since $\langle\xi,B\xi\rangle = b(B\xi,B\xi)=b(B^{\top}\xi,B^{\top}\xi)$ for any $\xi=(\xi_1,\dots,\xi_n)\in\mathbb{R}^{n}$, the following relations hold true
\begin{equation*}
\sum_{1\leq k,l\leq n}b_{kl}~\xi_{k}\xi_{l} = b(B\xi,B\xi)=b(B^{\top}\xi,B^{\top}\xi),
\end{equation*}
or equivalently,
\begin{equation}\label{FTer}
\sum_{1\leq k,l\leq n}b_{kl}~\pi_{k}(\xi)\pi_{l}(\xi)=b(B\xi,B\xi)=b(B^{\top}\xi,B^{\top}\xi),
\end{equation}
where $\pi_{j}(\xi_1,\dots,\xi_n)=\xi_{j}, ~ j\in\{1,\dots,n\}$, denotes the projection onto the $j$th factor. 

Recall now from Remark \ref{deltaa} the formula of the $b-$Laplace operator, namely,
\begin{align}\label{lapex}
\Delta_{b}=\sum_{1\leq k,l \leq n} b_{kl}~\partial_{k}\partial_{l}.
\end{align}
\begin{itemize}
\item[(i)] Let $\xi\in\mathbb{R}^{n}$ be arbitrary fixed. Using the relation \eqref{relb}, for any $\mathbf{x}=(x_1,\dots,x_n)\in\mathbb{R}^{n}$ and $j\in\{1,\dots,n\}$, we have
\begin{align*}
\dfrac{\partial}{\partial x_j}e^{2\pi i b(\mathbf{x},\xi)}&=2\pi i ~e^{2\pi i b(\mathbf{x},\xi)}\dfrac{\partial}{\partial x_j} b(\mathbf{x},\xi)=2\pi i ~ e^{2\pi i b(\mathbf{x},\xi)}\dfrac{\partial}{\partial x_j}\langle \mathbf{x},B^{-1}\xi\rangle\\
&=2\pi i~e^{2\pi i b(\mathbf{x},\xi)}~\pi_{j}(B^{-1}\xi).
\end{align*}
Thus for any $\mathbf{x}=(x_1,\dots,x_n)\in\mathbb{R}^{n}$ and $k,l\in\{1,\dots,n\}$, the following equality holds
\begin{align}\label{klder}
\dfrac{\partial^{2}}{\partial{x_k}\partial{x_l}}e^{2\pi i b(\mathbf{x},\xi)}=-4\pi^{2}~e^{2\pi i b(\mathbf{x},\xi)}~\pi_{k}(B^{-1}\xi)\pi_{l}(B^{-1}\xi).
\end{align}
Consequently, for any arbitrary fixed $\xi\in\mathbb{R}^{n}$, using \eqref{FTer}, \eqref{lapex}, and \eqref{klder}, we obtain 
\begin{align*}\label{lapexpL}
\Delta_{b}e^{2\pi i b(\mathbf{x},\xi)}&=-4\pi^{2}~e^{2\pi i b(\mathbf{x},\xi)}\sum_{1\leq k,l\leq n}b_{kl}~\pi_{k}(B^{-1}\xi)\pi_{l}(B^{-1}\xi)\\
&=[-4\pi^{2} b(\xi,\xi)]\cdot e^{ 2\pi i b(\mathbf{x},\xi)}, ~\forall \mathbf{x}\in\mathbb{R}^{n}.
\end{align*}

Now replacing $\xi$ by $-\xi$ in the relation
$$
\Delta_{b}e^{ 2\pi i b(\mathbf{x},\xi)}=[-4\pi^{2} b(\xi,\xi)]\cdot e^{ 2\pi i b(\mathbf{x},\xi)}, ~\forall \mathbf{x}\in\mathbb{R}^{n},
$$
and taking into account the bilinearity of $b$, we get
$$
\Delta_{b}e^{ -2\pi i b(\mathbf{x},\xi)}=[-4\pi^{2} b(\xi,\xi)]\cdot e^{ -2\pi i b(\mathbf{x},\xi)}, ~\forall \mathbf{x}\in\mathbb{R}^{n}.
$$
\item[(ii)] Let $\xi\in\mathbb{R}^{n}$ be arbitrary fixed. Using the relation \eqref{relb}, for any $\mathbf{x}=(x_1,\dots,x_n)\in\mathbb{R}^{n}$ and $j\in\{1,\dots,n\}$, we have
\begin{align*}
\dfrac{\partial}{\partial x_j}e^{2\pi i b(\xi,\mathbf{x})}&=2\pi i ~e^{2\pi i b(\xi,\mathbf{x})}\dfrac{\partial}{\partial x_j} b(\xi,\mathbf{x})=2\pi i ~ e^{2\pi i b(\xi,\mathbf{x})}\dfrac{\partial}{\partial x_j}\langle \xi,B^{-1}\mathbf{x}\rangle=\\
&=2\pi i ~ e^{2\pi i b(\xi,\mathbf{x})}\dfrac{\partial}{\partial x_j}\langle B^{-\top}\xi,\mathbf{x}\rangle=2\pi i~e^{2\pi i b(\xi,\mathbf{x})}~\pi_{j}(B^{-\top}\xi).
\end{align*}
Using the same type or arguments as in the proof of $(i)$, for any arbitrary fixed $\xi\in\mathbb{R}^{n}$, we obtain 
\begin{align*}
\Delta_{b}e^{2\pi i b(\xi,\mathbf{x})}&=-4\pi^{2}~e^{2\pi i b(\xi,\mathbf{x})}\sum_{1\leq k,l\leq n}b_{kl}~\pi_{k}(B^{-\top}\xi)\pi_{l}(B^{-\top}\xi)\\
&=[-4\pi^{2} b(\xi,\xi)]\cdot e^{ 2\pi i b(\xi,\mathbf{x})}, ~\forall \mathbf{x}\in\mathbb{R}^{n}.
\end{align*}
Replacing $\xi$ by $-\xi$ in the relation
$$
\Delta_{b}e^{ 2\pi i b(\xi,\mathbf{x})}=[-4\pi^{2} b(\xi,\xi)]\cdot e^{ 2\pi i b(\xi,\mathbf{x})}, ~\forall \mathbf{x}\in\mathbb{R}^{n},
$$
and using the bilinearity of $b$, we get
$$
\Delta_{b}e^{ -2\pi i b(\xi,\mathbf{x})}=[-4\pi^{2} b(\xi,\xi)]\cdot e^{ -2\pi i b(\xi,\mathbf{x})}, ~\forall \mathbf{x}\in\mathbb{R}^{n}.
$$
\end{itemize}
\end{proof}

Next result gives a very important compatibility between the $m-$th order $b-$Laplace operator and the left/right Fourier transforms associated to the geometric structure $b$. More precisely, it shows that the left/right Fourier transform maps the $m-$th order $b-$Laplacian of a Schwartz function to the product between the left/right Fourier transform of that function, and the polynomial function $[-4\pi^{2} b(\cdot,\cdot)]^{m}$. The identities given in the next result lead to the definition of the geometric fractional Laplacian, introduced and analyzed in the last section of this article. 
\begin{theorem}\label{mLap}
Let $b$ be a geometric structure on $\mathbb{R}^n$. Then for any $f\in\mathcal{S}(\mathbb{R}^{n})$ and $m\in\mathbb{N}$ we have
\begin{itemize}
\item[(i)] $[\mathcal{F}^{L/R}_{b}(\Delta_{b}^{m}f)](\xi)=[-4\pi^{2} b(\xi,\xi)]^{m} (\mathcal{F}^{L/R}_{b} f)(\xi), ~\forall \xi\in\mathbb{R}^{n}$,
\item[(ii)] $[(\mathcal{F}^{L/R}_{b})^{-1}(\Delta_{b}^{m}f)](\mathbf{x})=[-4\pi^{2} b(\mathbf{x},\mathbf{x})]^{m} [(\mathcal{F}^{L/R}_{b})^{-1}f](\mathbf{x}), ~\forall \mathbf{x}\in\mathbb{R}^{n}$.
\end{itemize}
\end{theorem}
\begin{proof}
\begin{itemize}
\item[(i)] We start be pointing out that it is enough to prove the required identities for $m=1$ (for $m=0$ both identities are obviously true), the general case following directly by induction on $m$. First, we shall prove a similar identity regarding the classical Fourier transform. In order to do that, let $(b,B)$ be the geometric pair associated to $b$ (see \eqref{relb}), and let $b_{ij}$, $i,j\in\{1,\dots,n\}$ be the entries of the matrix representing $B$ with respect to canonical basis. Using the standard formula below for the classical Fourier transform
\begin{align*}
[\mathcal{F}(\partial_{i}\partial_{j}f)](\xi)=-4\pi^{2}\xi_{i}\xi_{j}(\mathcal{F}f)(\xi), ~\forall \xi=(\xi_1,\dots,\xi_n)\in\mathbb{R}^{n},
\end{align*}
the definition of $b-$Laplace operator, the relations \eqref{FTer}, and the linearity of the Fourier transform, we get for any $\xi\in\mathbb{R}^{n}$, the following equalities
\begin{align*}
[\mathcal{F}(\Delta_{b}f)](\xi)&=[\mathcal{F}(\sum_{1\leq i,j\leq n} b_{ij}\partial_{i}\partial_{j}f)](\xi)=\sum_{1\leq i,j\leq n} b_{ij}[\mathcal{F}(\partial_{i}\partial_{j}f)](\xi)\\
&=-4\pi^{2}(\mathcal{F}f)(\xi)\sum_{1\leq i,j\leq n} b_{ij}\xi_{i}\xi_{j}\\
&=-4\pi^{2}b(B\xi,B\xi)~(\mathcal{F}f)(\xi)=-4\pi^{2}b(B^{\top}\xi,B^{\top}\xi)~(\mathcal{F}f)(\xi).
\end{align*}
Consequently, for any arbitrary fixed $f\in\mathcal{S}(\mathbb{R}^{n})$, we obtained 
\begin{equation}\label{RUtiL}
[\mathcal{F}(\Delta_{b}f)](\xi)=-4\pi^{2}b(B\xi,B\xi)(\mathcal{F}f)(\xi), ~\forall \xi\in\mathbb{R}^{n},
\end{equation}
\begin{equation}\label{RutiR}
[\mathcal{F}(\Delta_{b}f)](\xi)=-4\pi^{2}b(B^{\top}\xi,B^{\top}\xi)(\mathcal{F}f)(\xi), ~\forall \xi\in\mathbb{R}^{n}.
\end{equation}
Now we have all ingredients needed in order to prove the required identities for $m=1$. We start with the identity concerning the left Fourier transform, i.e.,
\begin{align}\label{yuoL}
[\mathcal{F}^{L}_{b}(\Delta_{b}f)](\xi)=-4\pi^{2} b(\xi,\xi) (\mathcal{F}^{L}_{b} f)(\xi), ~\forall \xi\in\mathbb{R}^{n}.
\end{align}
In order to prove this relation, let us fix an arbitrary $f\in\mathcal{S}(\mathbb{R}^{n})$. Then using the identity $\mathcal{F}^{L}_{b}=\tau_{B^{\top}}\circ\mathcal{F}$ (cf. Theorem \ref{teo2} $(i)$) and the relation \eqref{RutiR}, we have for any $\xi\in\mathbb{R}^{n}$
\begin{align*}
[\mathcal{F}^{L}_{b}(\Delta_{b}f)](\xi)&=[\tau_{B^{\top}}(\mathcal{F}(\Delta_{b}f))](\xi)=[\mathcal{F}(\Delta_{b}f)](B^{-\top}\xi)\\
&=-4\pi^{2}b(B^{\top}B^{-\top}\xi,B^{\top}B^{-\top}\xi)(\mathcal{F}f)(B^{-\top}\xi)\\
&=-4\pi^{2}b(\xi,\xi)(\mathcal{F}^{L}_{b}f)(\xi),
\end{align*}
and thus we get \eqref{yuoL}.

Let us now prove the similar identity for the right Fourier transform, i.e., 
\begin{align}\label{yuoR}
[\mathcal{F}^{R}_{b}(\Delta_{b}f)](\xi)=-4\pi^{2} b(\xi,\xi) (\mathcal{F}^{R}_{b} f)(\xi), ~\forall \xi\in\mathbb{R}^{n}.
\end{align}
In order to prove \eqref{yuoR}, let $f\in\mathcal{S}(\mathbb{R}^{n})$ be arbitrary fixed. Using the identity $\mathcal{F}^{R}_{b}=\tau_{B}\circ\mathcal{F}$ (cf. Theorem \ref{teo2} $(ii)$) and the relation \eqref{RUtiL}, we have for any $\xi\in\mathbb{R}^{n}$
\begin{align*}
[\mathcal{F}^{R}_{b}(\Delta_{b}f)](\xi)&=[\tau_{B}(\mathcal{F}(\Delta_{b}f))](\xi)=[\mathcal{F}(\Delta_{b}f)](B^{-1}\xi)\\
&=-4\pi^{2}b(BB^{-1}\xi,BB^{-1}\xi)(\mathcal{F}f)(B^{-1}\xi)\\
&=-4\pi^{2}b(\xi,\xi)(\mathcal{F}^{R}_{b}f)(\xi),
\end{align*}
hence we obtained \eqref{yuoR}.

As already mentioned before, the proof of the general case follows now trivially by induction on $m$.

\item[(ii)] The required identities are obtained by simply applying previous item for the geometric structure $-b$, then using the relation $\mathcal{F}^{R/L}_{-b}=\dfrac{1}{|\det b|}\cdot(\mathcal{F}^{L/R}_{b})^{-1}$ (cf. Theorem \ref{teo2} $(vii)$), and the identity $\Delta_{-b}^{m}=(-1)^{m}\cdot\Delta_{b}^{m}$.
\end{itemize}
\end{proof}

Given an arbitrary fixed $f\in\mathcal{S}(\mathbb{R}^{n})$, by applying the Theorem \ref{mLap} for the function $\tau_{A}f\in\mathcal{S}(\mathbb{R}^{n})$, with $A\in G_b$, and using the Theorem \ref{ginvFT}, we obtain the following result.
\begin{corollary}
Let $b$ be a geometric structure on $\mathbb{R}^n$. Then for any $f\in\mathcal{S}(\mathbb{R}^{n})$, $m\in\mathbb{N}$ and $A\in G_b$ we have
\begin{align*}
(\mathcal{F}^{L/R}_{b})^{\pm 1}[\Delta_{b}^{m}(\tau_{A}f)]=[-4\pi^{2} b(\cdot,\cdot)]^{m} \tau_{A}[(\mathcal{F}^{L/R}_{b})^{\pm 1}f].
\end{align*}
\end{corollary}
Next we provide an important result which relates the $L^2$ norm of the $m-$th order $b-$Laplace operator of a Schwartz function, with the $L^2$ norm of the left/right Fourier transform of that function multiplied by the polynomial function $b(\cdot,\cdot)^{m}$. More precisely, the following result holds.
\begin{theorem}\label{RTR}
Let $b$ be a geometric structure on $\mathbb{R}^n$. Then for any $f\in\mathcal{S}(\mathbb{R}^{n})$ and $m\in\mathbb{N}$ the following assertion holds
\begin{equation*}
\int_{\mathbb{R}^{n}}|(\Delta_{b}^{m}f)(\mathbf{x})|^{2}\mathrm{d}\mathbf{x}=(2\pi)^{4m}|\det b|\int_{\mathbb{R}^{n}} b(\xi,\xi)^{2m} \cdot |(\mathcal{F}^{L/R}_{b}f)(\xi)|^{2}\mathrm{d}\xi.
\end{equation*}
\end{theorem}
\begin{proof}
Using the identity from Proposition \ref{L2} $(iii)$ for the function $\Delta^{m}_{b}f\in\mathcal{S}(\mathbb{R}^{n})$, followed by Theorem \ref{mLap} we have
\begin{align*}
\int_{\mathbb{R}^{n}}|(\Delta_{b}^{m}f)(\mathbf{x})|^{2}\mathrm{d}\mathbf{x}&=|\det b|\int_{\mathbb{R}^{n}}|[\mathcal{F}^{L/R}_{b}(\Delta^{m}_{b}f)](\xi)|^{2}\mathrm{d}\xi\\
&=(2\pi)^{4m}|\det b|\int_{\mathbb{R}^{n}} b(\xi,\xi)^{2m} \cdot |(\mathcal{F}^{L/R}_{b}f)(\xi)|^{2}\mathrm{d}\xi,
\end{align*}
and thus we obtain the conclusion.
\end{proof}

\section{The geometric Fourier transforms and Poisson's summation formula for full lattices}

A natural environment where the geometric Fourier transforms appear is the Poisson summation formula for full lattices of $\mathbb{R}^{n}$. In order to be precise we shall introduce first some terminology. More exactly, recall that a \textit{full lattice} of $\mathbb{R}^{n}$ is a set $\mathcal{L}\subset\mathbb{R}^{n}$ generated by $\mathbb{Z}^{n}$ and some automorphism $B\in\operatorname{Aut}(\mathbb{R}^{n})$, i.e., $\mathcal{L}=B(\mathbb{Z}^{n})$. Moreover, denoting by $\langle\cdot,\cdot\rangle$ the canonical Euclidean inner product on $\mathbb{R}^{n}$, the set $\mathcal{L}^{\star}:=\{\mathbf{x}\in\mathbb{R}^{n}:~\langle\mathbf{x},n\rangle\in\mathbb{Z}, ~\forall n\in\mathcal{L}\}$ is also a full lattice, called the \textit{dual lattice} associated to $\mathcal{L}$. Recall that the dual lattice associated to $\mathcal{L}=B(\mathbb{Z}^{n})$ can by equivalently written as $\mathcal{L}^{\star}=B^{-\top}(\mathbb{Z}^{n})$. Given a lattice $\mathcal{L}=B(\mathbb{Z}^{n})$, the $n-$dimensional volume of the \textit{fundamental parallelepiped} $\mathcal{P}(B):=B([0,1)^{n})$, also called the \textit{determinant of the lattice} $\mathcal{L}$, is given by $\det{\mathcal{L}}=|\det B|$. 

Now we have all ingredients needed in order to state the classical Poisson summation formula for full lattices (see for details, e.g., \cite{SR}). More exactly, given a full lattice $\mathcal{L}\subset\mathbb{R}^{n}$ and the associated dual lattice $\mathcal{L}^{\star}$, the following relation holds true for any $f\in\mathcal{S}(\mathbb{R}^{n})$ and $\mathbf{x}\in\mathbb{R}^{n}$
\begin{equation}\label{PSF}
\sum_{\mathbf{n}\in\mathcal{L}}f(\mathbf{n}+\mathbf{x})=\dfrac{1}{\det\mathcal{L}}\cdot\sum_{\mathbf{m}\in\mathcal{L}^{\star}}(\mathcal{F}f)(\mathbf{m})\cdot e^{2\pi i\langle\mathbf{x},\mathbf{m}\rangle}.
\end{equation}

If we write the sum from the right hand side of \eqref{PSF} as a sum running over the elements of the canonical lattice $\mathbb{Z}^{n}$, it shows up naturally the left/right Fourier transform associated to the geometric structure given by $b(\mathbf{x},\mathbf{y}):=\langle\mathbf{x},B^{-1}\mathbf{y}\rangle, ~\forall\mathbf{x},\mathbf{y}\in\mathbb{R}^{n}$, where  $B\in\operatorname{Aut}(\mathbb{R}^{n})$ generates the full lattice $\mathcal{L}$, i.e., $\mathcal{L}=B(\mathbb{Z}^{n})$. More precisely, we have the following result.

\begin{theorem}\label{SUPTh}
Let $\mathcal{L}=B(\mathbb{Z}^{n})$ be a full lattice and let $b$ be the geometric structure given by $b(\mathbf{x},\mathbf{y}):=\langle\mathbf{x},B^{-1}\mathbf{y}\rangle, ~\forall\mathbf{x},\mathbf{y}\in\mathbb{R}^{n}$. Then for any $f\in\mathcal{S}(\mathbb{R}^{n})$, and $\mathbf{x}\in\mathbb{R}^{n}$, the following relations hold
\begin{align}\label{PoiL}
\sum_{\mathbf{n}\in\mathcal{L}}f(\mathbf{n}+\mathbf{x})=|\det b|\cdot\sum_{\xi\in\mathbb{Z}^{n}}(\mathcal{F}^{L}_{b}f)(\mathbf{\xi})\cdot e^{2\pi i b(\xi,\mathbf{x})},
\end{align}
\begin{align}\label{PoiR}
\sum_{\mathbf{n}\in\mathcal{L}}f(\mathbf{n}+\mathbf{x})=|\det b^{op}|\cdot\sum_{\xi\in\mathbb{Z}^{n}}(\mathcal{F}^{R}_{b^{op}}f)(\mathbf{\xi})\cdot e^{2\pi i b^{op}(\mathbf{x},\xi)},
\end{align}
where $b^{op}$ is the opposite geometric structure associated to $b$. Particularly, for $\mathbf{x}=\mathbf{0}$ the identities \eqref{PoiL}, \eqref{PoiR} become
\begin{align*}
\sum_{\mathbf{n}\in\mathcal{L}}f(\mathbf{n})=|\det b|\cdot\sum_{\xi\in\mathbb{Z}^{n}}(\mathcal{F}^{L}_{b}f)(\mathbf{\xi})=|\det b^{op}|\cdot\sum_{\xi\in\mathbb{Z}^{n}}(\mathcal{F}^{R}_{b^{op}}f)(\mathbf{\xi}).
\end{align*} 
\end{theorem}
\begin{proof}
Let us start by recalling that $\det\mathcal{L}=|\det B|=\dfrac{1}{|\det b|}$ and $\mathcal{L}^{\star}=B^{-\top}(\mathbb{Z}^{n})$. Using Poisson's summation formula, for any $f\in\mathcal{S}(\mathbb{R}^{n})$ and $\mathbf{x}\in\mathbb{R}^{n}$ we obtain successively 
\begin{align*}
\sum_{\mathbf{n}\in\mathcal{L}}f(\mathbf{n}+\mathbf{x})&=\dfrac{1}{\det\mathcal{L}}\cdot\sum_{\mathbf{m}\in\mathcal{L}^{\star}}(\mathcal{F}f)(\mathbf{m})\cdot e^{2\pi i\langle\mathbf{x},\mathbf{m}\rangle}=|\det b|\cdot \sum_{\xi\in\mathbb{Z}^{n}}(\mathcal{F}f)(B^{-\top}\xi)\cdot e^{2\pi i\langle\mathbf{x},B^{-\top}\xi\rangle}\\
&=|\det b|\cdot \sum_{\xi\in\mathbb{Z}^{n}}(\mathcal{F}^{L}_{b}f)(\xi)\cdot e^{2\pi i\langle B^{-1}\mathbf{x},\xi\rangle}=|\det b|\cdot \sum_{\xi\in\mathbb{Z}^{n}}(\mathcal{F}^{L}_{b}f)(\xi)\cdot e^{2\pi i\langle \xi, B^{-1}\mathbf{x}\rangle}\\
&=|\det b|\cdot\sum_{\xi\in\mathbb{Z}^{n}}(\mathcal{F}^{L}_{b}f)(\mathbf{\xi})\cdot e^{2\pi i b(\xi,\mathbf{x})},
\end{align*} 
thus we get the relation \eqref{PoiL}.

The identity \eqref{PoiR} follows directly from the relation \eqref{PoiL}, taking into account that $|\det b|=|\det b^{op}|$, and $\mathcal{F}^{L}_{b}=\mathcal{F}^{R}_{b^{op}}$, cf. Remark \ref{fop} $(i)$.
\end{proof}

In the hypothesis of Theorem \ref{SUPTh}, note that if $\mathcal{L}=B(\mathbb{Z}^{n})$ then $\mathcal{L}=(-B)(\mathbb{Z}^{n})$, and moreover $|\det (-b)|=|\det b|$. Thus, the identity $(\mathcal{F}^{L/R}_{b})^{-1}=|\det b|\cdot\mathcal{F}^{R/L}_{-b}$ (cf. Theorem \ref{teo2} $(vii)$) implies immediately that the relations given in Theorem \ref{SUPTh} can be reformulated in terms of the inverse left/right Fourier transforms. More precisely, the following result holds.
\begin{corollary}\label{corolul}
Let $\mathcal{L}=B(\mathbb{Z}^{n})$ be a full lattice and let $b$ be the geometric structure given by $b(\mathbf{x},\mathbf{y}):=\langle\mathbf{x},B^{-1}\mathbf{y}\rangle, ~\forall\mathbf{x},\mathbf{y}\in\mathbb{R}^{n}$. Then for any $f\in\mathcal{S}(\mathbb{R}^{n})$, and $\xi\in\mathbb{R}^{n}$, we have
\begin{align*}
\sum_{\mathbf{n}\in\mathcal{L}}f(\mathbf{n}+\xi)&=\sum_{\mathbf{x}\in\mathbb{Z}^{n}}[(\mathcal{F}^{R}_{b})^{-1}f](\mathbf{x})\cdot e^{-2\pi i b(\mathbf{x},\xi)},\\
\sum_{\mathbf{n}\in\mathcal{L}}f(\mathbf{n}+\xi)&=\sum_{\mathbf{x}\in\mathbb{Z}^{n}}[(\mathcal{F}^{L}_{b^{op}})^{-1}f](\mathbf{x})\cdot e^{-2\pi i b^{op}(\xi,\mathbf{x})}.
\end{align*}
Particularly, for $\xi=\mathbf{0}$ the above identities become
\begin{align*}
\sum_{\mathbf{n}\in\mathcal{L}}f(\mathbf{n})=\sum_{\mathbf{x}\in\mathbb{Z}^{n}}[(\mathcal{F}^{R}_{b})^{-1}f](\mathbf{x})=\sum_{\mathbf{x}\in\mathbb{Z}^{n}}[(\mathcal{F}^{L}_{b^{op}})^{-1}f](\mathbf{x}).
\end{align*}
\end{corollary}

So far we have shown that to an arbitrary given full lattice $\mathcal{L}=B(\mathbb{Z}^{n})$, one can naturally associate the geometric structure  $b(\mathbf{x},\mathbf{y})=\langle\mathbf{x},B^{-1}\mathbf{y}\rangle, ~\forall\mathbf{x},\mathbf{y}\in\mathbb{R}^{n}$, such that Poisson's summation formula for $\mathcal{L}$ can be written in terms of $b$ or $b^{op}$ and the associated left or right Fourier transform. 

Conversely, to any geometric structure $b$ on $\mathbb{R}^{n}$ one can associate two full lattices, $\mathcal{L}^{L}_{b}:=B(\mathbb{Z}^{n})$, and $\mathcal{L}^{R}_{b}:=B^{\top}(\mathbb{Z}^{n})$, where $B\in\operatorname{Aut}(\mathbb{R}^{n})$ generates the geometric pair associated to $b$, i.e., $\langle\mathbf{x},\mathbf{y}\rangle=b(\mathbf{x},B\mathbf{y}), \forall\mathbf{x},\mathbf{y}\in\mathbb{R}^{n}$. Note that when $b$ is a symmetric or a  skew--symmetric geometric structure, then $\mathcal{L}^{L}_{b}=\mathcal{L}^{R}_{b}$. 

Using the above notations, the following Poisson--like summation formulas hold for the lattices $\mathcal{L}^{L/R}_{b}$.

\begin{theorem}\label{PoisLR}
Let $b$ be a geometric structure on $\mathbb{R}^{n}$, and $(b,B)$ the associated geometric pair. Then for any $f\in\mathcal{S}(\mathbb{R}^{n})$ and $\mathbf{x}\in\mathbb{R}^{n}$ 
\begin{align}\label{PL}
\sum_{\mathbf{n}\in\mathcal{L}^{L}_{b}}f(\mathbf{n}+\mathbf{x})=|\det b|\cdot\sum_{\xi\in\mathbb{Z}^{n}}(\mathcal{F}^{L}_{b}f)(\mathbf{\xi})\cdot e^{2\pi i b(\xi,\mathbf{x})},
\end{align}
\begin{align}\label{PR}
\sum_{\mathbf{n}\in\mathcal{L}^{R}_{b}}f(\mathbf{n}+\mathbf{x})=|\det b|\cdot\sum_{\xi\in\mathbb{Z}^{n}}(\mathcal{F}^{R}_{b}f)(\mathbf{\xi})\cdot e^{2\pi i b(\mathbf{x},\xi)},
\end{align}
where $\mathcal{L}^{L}_{b}=B(\mathbb{Z}^{n})$, and $\mathcal{L}^{R}_{b}=B^{\top}(\mathbb{Z}^{n})$, are the left and right full lattices generated by the geometric pair $(b,B)$.
\end{theorem}
\begin{proof}
Let us start by noticing that $\det\mathcal{L}^{L}_{b}=\det\mathcal{L}^{R}_{b}=|\det B|=\dfrac{1}{|\det b|}$, and  $(\mathcal{L}^{L}_{b})^{\star}=B^{-\top}(\mathbb{Z}^{n})$, $(\mathcal{L}^{R}_{b})^{\star}=B^{-1}(\mathbb{Z}^{n})$.

Now using Poisson's summation formula for the lattice $\mathcal{L}^{L}_{b}$, for any $f\in\mathcal{S}(\mathbb{R}^{n})$ and $\mathbf{x}\in\mathbb{R}^{n}$ we have 
\begin{align*}
\sum_{\mathbf{n}\in\mathcal{L}^{L}_{b}}f(\mathbf{n}+\mathbf{x})&=\dfrac{1}{\det\mathcal{L}^{L}_{b}}\cdot\sum_{\mathbf{m}\in(\mathcal{L}^{L}_{b})^{\star}}(\mathcal{F}f)(\mathbf{m})\cdot e^{2\pi i\langle\mathbf{x},\mathbf{m}\rangle}=|\det b|\cdot\sum_{\xi\in\mathbb{Z}^{n}}(\mathcal{F}f)(B^{-\top}\xi)\cdot e^{2\pi i\langle\mathbf{x},B^{-\top}\xi\rangle}\\
&=|\det b|\cdot\sum_{\xi\in\mathbb{Z}^{n}}(\mathcal{F}^{L}_{b}f)(\xi)\cdot e^{2\pi i\langle B^{-1}\mathbf{x},\xi\rangle}=|\det b|\cdot\sum_{\xi\in\mathbb{Z}^{n}}(\mathcal{F}^{L}_{b}f)(\xi)\cdot e^{2\pi i\langle \xi,B^{-1}\mathbf{x}\rangle}\\
&=|\det b|\cdot\sum_{\xi\in\mathbb{Z}^{n}}(\mathcal{F}^{L}_{b}f)(\xi)\cdot e^{2\pi i b(\xi,\mathbf{x})},
\end{align*}
thus we get the relation \eqref{PL}.

The identity \eqref{PR} follows similarly, using Poisson's summation formula for the lattice $\mathcal{L}^{R}_{b}$. More precisely, for any $f\in\mathcal{S}(\mathbb{R}^{n})$ and $\mathbf{x}\in\mathbb{R}^{n}$ we have
\begin{align*}
\sum_{\mathbf{n}\in\mathcal{L}^{R}_{b}}f(\mathbf{n}+\mathbf{x})&=\dfrac{1}{\det\mathcal{L}^{R}_{b}}\cdot\sum_{\mathbf{m}\in(\mathcal{L}^{R}_{b})^{\star}}(\mathcal{F}f)(\mathbf{m})\cdot e^{2\pi i\langle\mathbf{x},\mathbf{m}\rangle}=|\det b|\cdot\sum_{\xi\in\mathbb{Z}^{n}}(\mathcal{F}f)(B^{-1}\xi)\cdot e^{2\pi i\langle\mathbf{x},B^{-1}\xi\rangle}\\
&=|\det b|\cdot\sum_{\xi\in\mathbb{Z}^{n}}(\mathcal{F}^{R}_{b}f)(\xi)\cdot e^{2\pi i b(\mathbf{x},\xi)},
\end{align*}
thus we obtained the relation \eqref{PR}.
\end{proof}

Using the same type of arguments as in Corollary \ref{corolul}, we obtain the following result which reformulates the relations given in Theorem \ref{PoisLR} in terms of the inverse left/right Fourier transforms.
\begin{corollary}
Let $b$ be a geometric structure on $\mathbb{R}^{n}$, and $(b,B)$ the associated geometric pair. Then for any $f\in\mathcal{S}(\mathbb{R}^{n})$ and $\xi\in\mathbb{R}^{n}$ 
\begin{align*}
\sum_{\mathbf{n}\in\mathcal{L}^{L}_{b}}f(\mathbf{n}+\xi)&=\sum_{\mathbf{x}\in\mathbb{Z}^{n}}[(\mathcal{F}^{R}_{b})^{-1}f](\mathbf{x})\cdot e^{-2\pi i b(\mathbf{x},\xi)},\\
\sum_{\mathbf{n}\in\mathcal{L}^{R}_{b}}f(\mathbf{n}+\xi)&=\sum_{\mathbf{x}\in\mathbb{Z}^{n}}[(\mathcal{F}^{L}_{b})^{-1}f](\mathbf{x})\cdot e^{-2\pi i b(\xi,\mathbf{x})},
\end{align*}
where $\mathcal{L}^{L}_{b}=B(\mathbb{Z}^{n})$, and $\mathcal{L}^{R}_{b}=B^{\top}(\mathbb{Z}^{n})$, are the left and right full lattices generated by the geometric pair $(b,B)$.
\end{corollary}

\begin{remark}
All results from this section still holds true if the functions we study belong to the \textit{Poisson space} instead of $\mathcal{S}(\mathbb{R}^{n})$. Recall that a function $f:\mathbb{R}^{n}\rightarrow \mathbb{C}$ belongs to the Poisson space (see for details, e.g., \cite{SR}) if $f\in L^{2}(\mathbb{R}^{n})$  and there exists $\delta, C >0$ such that for all $\mathbf{x}\in\mathbb{R}^{n}$
$$
|f(\mathbf{x})|<\dfrac{C}{(1+\|\mathbf{x}\|)^{n+\delta}}, ~\text{and} ~~ |(\mathcal{F}f)(\mathbf{x})|<\dfrac{C}{(1+\|\mathbf{x}\|)^{n+\delta}},
$$
where by $\|\cdot\|$ we have denoted the norm induced by the canonical Euclidean inner product of $\mathbb{R}^{n}$.
Note that Schwartz functions are obviously elements of the Poisson space.
\end{remark}

\section{A geometric fractional Laplacian}

The aim of this section is to introduce a fractional Laplacian naturally associated to each pair $(b,s)$ consisting of a positive definite geometric structure $b$, and a real number $s\in (0,1)$. In order to do so, notice first that \textit{if $b$ is a positive definite geometric structure} on $\mathbb{R}^n$ (i.e., $b(\mathbf{x},\mathbf{x})>0, ~\forall \mathbf{x}\in\mathbb{R}^{n}\setminus\{\mathbf{0}\}$), then for any fixed $s\in (0,1)$, one can define two natural pseudo-differential operators, which will be called the \textit{left/right geometric fractional Laplacian}, denoted by $(-\Delta_{b})_{L/R}^{s}$, and given for any $f\in\mathcal{S}(\mathbb{R}^{n})$ by the formula
\begin{equation}\label{fLR}
[\mathcal{F}^{L/R}_{b}((-\Delta_{b})_{L/R}^{s}f)](\xi)=[4\pi^{2} b(\xi,\xi)]^{s}\cdot (\mathcal{F}^{L/R}_{b} f)(\xi), ~\forall \xi\in\mathbb{R}^{n}.
\end{equation}
We will show that for any fixed $s\in (0,1)$, the left and right fractional geometric Laplacian actually coincide, i.e., $(-\Delta_{b})_{L}^{s}=(-\Delta_{b})_{R}^{s}$, thus leading to the existence of a \textit{geometric fractional Laplacian} (denoted by $(-\Delta_{b})^{s}$), regardless the symmetry--like properties of the geometric structure $b$.

Before doing so, let us notice that for $f\in\mathcal{S}(\mathbb{R}^{n})$ the functions $\xi\mapsto [\Phi^{L/R}_{b}f](\xi):=[4\pi^2 b(\xi,\xi)]^{s}\cdot(\mathcal{F}^{L/R}_{b}f)(\xi)$ do not belong to $\mathcal{S}(\mathbb{R}^{n})$ due to the singularity in $\xi=0$ created by $\xi\mapsto[4\pi^2 b(\xi,\xi)]^{s}$. Nevertheless, for $f\in\mathcal{S}(\mathbb{R}^{n})$, the functions $\Phi^{L/R}_{b}f$ belong to $L^{1}(\mathbb{R}^{n})\cap L^{2}(\mathbb{R}^{n})$, since $[\tau_{B^{-\top}}\Phi^{L}_{b}f](\xi)=[\tau_{B^{-1}}\Phi^{R}_{b}f](\xi)=[4\pi^2\langle\xi,B\xi\rangle]^{s}\cdot (\mathcal{F}f)(\xi), ~\forall \xi\in\mathbb{R}^{n}$, where $(b,B)$ is the geometric pair generated by $b$, and $\mathcal{F}$ is the classical Fourier transform. Thus, the left/right inverse Fourier transform of $\Phi^{L/R}_{b}f$ is an element of $L^{2}(\mathbb{R}^{n})$, and consequently, the left/right geometric fractional Laplacian of a general $f\in\mathcal{S}(\mathbb{R}^{n})$ belongs to $L^{2}(\mathbb{R}^{n})$ -- here we considered the natural extensions to $L^{2}(\mathbb{R}^{n})$ of the left/right geometric Fourier transforms and the corresponding inverses. 

Throughout this section, by left/right geometric Fourier (and inverse Fourier) transforms we mean their appropriate extensions.

\begin{theorem}\label{LeqR}
Let $b$ be a positive definite geometric structure on $\mathbb{R}^n$ and $s\in (0,1)$. Then $(-\Delta_{b})_{L}^{s}=(-\Delta_{b})_{R}^{s}$.
\end{theorem}
\begin{proof}
Let $(b,B)$ be the geometric pair associated to the geometric structure $b$. For any fixed $s\in (0,1)$, let us denote by $g^{s}_{b}:\mathbb{R}^{n}\rightarrow\mathbb{R}$ the function given by $g^{s}_{b}(\xi):=[4\pi^2 b(\xi,\xi)]^{s}, ~\forall \xi\in\mathbb{R}^{n}$. In order to prove that $(-\Delta_{b})_{L}^{s}=(-\Delta_{b})_{R}^{s}$ recall first from \eqref{fLR} that for any $f\in\mathcal{S}(\mathbb{R}^{n})$
\begin{equation*}
(-\Delta_{b})_{L}^{s}f=(\mathcal{F}^{L}_{b})^{-1}[g^{s}_{b}\cdot(\mathcal{F}^{L}_{b}f)],~~~
(-\Delta_{b})_{R}^{s}f=(\mathcal{F}^{R}_{b})^{-1}[g^{s}_{b}\cdot(\mathcal{F}^{R}_{b}f)].
\end{equation*}
Consequently, we need to prove that for any $f\in\mathcal{S}(\mathbb{R}^{n})$
\begin{equation}\label{flR}
(\mathcal{F}^{L}_{b})^{-1}[g^{s}_{b}\cdot(\mathcal{F}^{L}_{b}f)]=(\mathcal{F}^{R}_{b})^{-1}[g^{s}_{b}\cdot(\mathcal{F}^{R}_{b}f)],
\end{equation}
or equivalently
\begin{equation*}
[\mathcal{F}^{R}_{b}\circ(\mathcal{F}^{L}_{b})^{-1}][g^{s}_{b}\cdot(\mathcal{F}^{L}_{b}f)]=g^{s}_{b}\cdot(\mathcal{F}^{R}_{b}f).
\end{equation*}
By Theorem \ref{teo2}, $(iii)$, the above relation is equivalent to
\begin{align*}
\tau_{BB^{-\top}}[g^{s}_{b}\cdot(\mathcal{F}^{L}_{b}f)]=g^{s}_{b}\cdot(\mathcal{F}^{R}_{b}f) &\Leftrightarrow \tau_{B^{-\top}}[g^{s}_{b}\cdot(\mathcal{F}^{L}_{b}f)]=\tau_{B^{-1}}[g^{s}_{b}\cdot(\mathcal{F}^{R}_{b}f)]\\
\Leftrightarrow\tau_{B^{-\top}}g^{s}_{b}\cdot\tau_{B^{-\top}}(\mathcal{F}^{L}_{b}f)&=\tau_{B^{-1}}g^{s}_{b}\cdot\tau_{B^{-1}}(\mathcal{F}^{R}_{b}f).
\end{align*}
As for any $\xi\in\mathbb{R}^n$, $(\tau_{B^{-\top}}g^{s}_{b})(\xi)=[4\pi^2b(B^{\top}\xi,B^{\top}\xi)]^{s}=[4\pi^2b(B\xi,B\xi)]^{s}=(\tau_{B^{-1}}g^{s}_{b})(\xi)$, the relation 
\begin{equation}\label{uu}
\tau_{B^{-\top}}g^{s}_{b}\cdot\tau_{B^{-\top}}(\mathcal{F}^{L}_{b}f)=\tau_{B^{-1}}g^{s}_{b}\cdot\tau_{B^{-1}}(\mathcal{F}^{R}_{b}f)
\end{equation}
is equivalent (on $\mathbb{R}^{n}\setminus\{\mathbf{0}\}$) to $\tau_{B^{-\top}}(\mathcal{F}^{L}_{b}f)=\tau_{B^{-1}}(\mathcal{F}^{R}_{b}f)$, which is the same as the item $(iii)$ from Theorem \ref{teo2}. Evaluated at the origin, the relation \eqref{uu} is obviously true, since $(\tau_{B^{-\top}}g^{s}_{b})(\mathbf{0})=(\tau_{B^{-1}}g^{s}_{b})(\mathbf{0})=0$. Consequently, the relation \eqref{flR} is true for any $f\in\mathcal{S}(\mathbb{R}^{n})$, and hence we obtained the conclusion.
\end{proof}

Therefore we have shown that for any positive definite geometric structure $b$, and $s\in (0,1)$, there exists a geometric fractional Laplacian, denoted by $(-\Delta_b)^{s}:=(-\Delta_b)^{s}_{L}=(-\Delta_b)^{s}_{R}$, and given for any $f\in\mathcal{S}(\mathbb{R}^{n})$ by the relation
\begin{equation}\label{definEq}
[\mathcal{F}^{L/R}_{b}((-\Delta_{b})^{s}f)](\xi)=[4\pi^{2} b(\xi,\xi)]^{s}\cdot (\mathcal{F}^{L/R}_{b} f)(\xi), ~\forall \xi\in\mathbb{R}^{n}.
\end{equation}
An interesting special case is obtained for $s=1/2$, introducing the pseudo-differential operator $\sqrt{-\Delta_{b}}$, defined for any $f\in\mathcal{S}(\mathbb{R}^{n})$ by the formula
\begin{equation*}
[\mathcal{F}^{L/R}_{b}(\sqrt{-\Delta_{b}}f)](\xi)=2\pi \sqrt{b(\xi,\xi)}\cdot (\mathcal{F}^{L/R}_{b} f)(\xi), ~\forall \xi\in\mathbb{R}^{n}.
\end{equation*}

Next we provide an explicit formulation of the geometric fractional Laplacian associated to a positive definite geometric structure $b$, as a pseudo-differential operator of standard type.
\begin{theorem}\label{clasFL}
Let $b$ be a positive definite geometric structure on $\mathbb{R}^{n}$, $(b,B)$ the associated geometric pair, and $s\in (0,1)$. Then for any $f\in\mathcal{S}(\mathbb{R}^{n})$ we have
\begin{align*}
[(-\Delta_{b})^{s}f](\mathbf{x})=(4\pi^2)^{s}\cdot\int_{\mathbb{R}^{n}}e^{2\pi i\langle \mathbf{x},\zeta \rangle}\cdot\langle\zeta,B\zeta\rangle^{s}\cdot(\mathcal{F}f)(\zeta)\mathrm{d}\zeta,
\end{align*}
where $\langle\cdot,\cdot\rangle$ is the canonical Euclidean inner product on $\mathbb{R}^{n}$, and $\mathcal{F}$ is the classical Fourier transform.
\end{theorem}
\begin{proof}
Let $f\in\mathcal{S}(\mathbb{R}^{n})$ be arbitrary fixed. Then from Theorem \ref{LeqR} and the relation \eqref{left}, we obtain for any $\mathbf{x}\in\mathbb{R}^{n}$
\begin{align*}
[(-\Delta_{b})^{s}f](\mathbf{x})&=[(-\Delta_{b})_{L}^{s}f](\mathbf{x})=[(\mathcal{F}^{L}_{b})^{-1}[(4\pi^{2}b(\cdot,\cdot))^{s}(\mathcal{F}^{L}_{b}f)]](\mathbf{x})\\
&=|\det b|\cdot \int_{\mathbb{R}^{n}}e^{2\pi i b(\xi,\mathbf{x})}\cdot [4\pi^{2}b(\xi,\xi)]^{s}\cdot (\mathcal{F}^{L}_{b}f)(\xi)\mathrm{d}\xi\\
&=(4\pi^2)^{s}|\det b|\cdot \int_{\mathbb{R}^{n}}e^{2\pi i b(\xi,\mathbf{x})}\cdot b(\xi,\xi)^{s}\cdot (\mathcal{F}f)(B^{-\top}\xi)\mathrm{d}\xi\\
&=(4\pi^2)^{s}|\det b|\cdot \int_{\mathbb{R}^{n}}e^{2\pi i b(B^{\top}\zeta,\mathbf{x})}\cdot b(B^{\top}\zeta,B^{\top}\zeta)^{s}\cdot (\mathcal{F}f)(\zeta) \cdot |\det B^{\top}|\mathrm{d}\zeta\\
&=(4\pi^2)^{s}|\det b|\cdot \int_{\mathbb{R}^{n}}e^{2\pi i b(B^{\top}\zeta,\mathbf{x})}\cdot b(B^{\top}\zeta,B^{\top}\zeta)^{s}\cdot (\mathcal{F}f)(\zeta) \cdot |\det b|^{-1}\mathrm{d}\zeta\\
&=(4\pi^2)^{s}\cdot \int_{\mathbb{R}^{n}}e^{2\pi i \langle B^{\top}\zeta,B^{-1}\mathbf{x}\rangle}\cdot \langle B^{\top}\zeta,B^{-1}B^{\top}\zeta\rangle^{s}\cdot(\mathcal{F}f)(\zeta)\mathrm{d}\zeta\\
&=(4\pi^2)^{s}\cdot \int_{\mathbb{R}^{n}}e^{2\pi i \langle \zeta,\mathbf{x}\rangle}\cdot \langle \zeta,B^{\top}\zeta\rangle^{s}\cdot(\mathcal{F}f)(\zeta)\mathrm{d}\zeta\\
&=(4\pi^2)^{s}\cdot \int_{\mathbb{R}^{n}}e^{2\pi i \langle \mathbf{x},\zeta \rangle}\cdot \langle B\zeta,\zeta\rangle^{s}\cdot(\mathcal{F}f)(\zeta)\mathrm{d}\zeta \\
&=(4\pi^2)^{s}\cdot\int_{\mathbb{R}^{n}}e^{2\pi i\langle \mathbf{x},\zeta\rangle}\cdot\langle\zeta,B\zeta\rangle^{s}\cdot(\mathcal{F}f)(\zeta)\mathrm{d}\zeta,
\end{align*}
and thus we get the conclusion.
\end{proof}

\begin{remark}
\begin{itemize}
\item[(i)] Note that any pseudo-differential operator with symbol of type $$P(\zeta)=[4\pi^{2}\cdot\langle\zeta,B\zeta\rangle]^{s}, ~\forall\zeta\in\mathbb{R}^{n},$$ for some positive definite $B\in\operatorname{Aut}(\mathbb{R}^{n})$, and $s\in(0,1)$, is actually the geometric fractional Laplacian generated by the geometric structure $$b(\mathbf{x},\mathbf{y}):=\langle\mathbf{x},B^{-1}\mathbf{y}\rangle, ~\forall\mathbf{x},\mathbf{y}\in\mathbb{R}^{n}.$$
\item[(ii)] In the case when the geometric structure $b$ is the canonical Euclidean inner product on $\mathbb{R}^{n}$ (i.e., $B=I_{d}$), the geometric fractional Laplacian coincides with the classical fractional Laplacian, up to a multiplicative constant. 
\item[(iii)] The approach to prove Theorem \ref{clasFL}, together with a similar one for the right geometric fractional Laplacian, lead to an alternative (analytic) proof of Theorem \ref{LeqR}, by showing that for any $f\in\mathcal{S}(\mathbb{R}^{n})$
\begin{align*}
[(-\Delta_{b})_{L}^{s}f](\mathbf{x})&=[(\mathcal{F}^{L}_{b})^{-1}[(4\pi^{2}b(\cdot,\cdot))^{s}(\mathcal{F}^{L}_{b}f)]](\mathbf{x})\\
&=(4\pi^2)^{s}\cdot\int_{\mathbb{R}^{n}}e^{2\pi i\langle \mathbf{x},\xi\rangle}\cdot\langle\xi,B\xi\rangle^{s}\cdot(\mathcal{F}f)(\xi)\mathrm{d}\xi\\
&=[(\mathcal{F}^{R}_{b})^{-1}[(4\pi^{2}b(\cdot,\cdot))^{s}(\mathcal{F}^{R}_{b}f)]](\mathbf{x})=[(-\Delta_{b})_{R}^{s}f](\mathbf{x}), ~\forall\mathbf{x}\in\mathbb{R}^{n}.
\end{align*}
\item[(iv)] Moreover, computations similar to those in the proof of Theorem \ref{clasFL}, show that for any $f\in\mathcal{S}(\mathbb{R}^{n})$
\begin{align*}
(\mathcal{F}^{L}_{b}([4\pi^{2}b(\cdot,\cdot)]^{s}[(\mathcal{F}^{L}_{b})^{-1}f]))(\mathbf{x})&=(\mathcal{F}^{R}_{b}([4\pi^{2}b(\cdot,\cdot)]^{s}[(\mathcal{F}^{R}_{b})^{-1}f]))(\mathbf{x})\\
&=(4\pi^2)^{s}\cdot\int_{\mathbb{R}^{n}}e^{2\pi i\langle \mathbf{x},\xi\rangle}\cdot\langle\xi,B\xi\rangle^{s}\cdot(\mathcal{F}f)(\xi)\mathrm{d}\xi\\
&=[(-\Delta_{b})^{s}f](\mathbf{x}), ~\forall\mathbf{x}\in\mathbb{R}^{n},
\end{align*}
thereby leading to the identity
\begin{equation*}
[(\mathcal{F}^{L/R}_{b})^{-1}((-\Delta_{b})^{s}f)](\mathbf{x})=[4\pi^{2} b(\mathbf{x},\mathbf{x})]^{s}\cdot [(\mathcal{F}^{L/R}_{b})^{-1} f](\mathbf{x}), ~\forall \mathbf{x}\in\mathbb{R}^{n},
\end{equation*}
which thus turns out to be an alternative for the equality \eqref{definEq} defining the geometric fractional Laplacian.
\item[(v)] Notice that if $b$ is a positive definite geometric structure on $\mathbb{R}^n$, then for any $f\in\mathcal{S}(\mathbb{R}^{n})$ and $s\in(0,1)$, using the same type of arguments as in the proof of Theorem \ref{RTR} it follows that
\begin{equation*}
\int_{\mathbb{R}^{n}}|((-\Delta_{b})^{s}f)(\mathbf{x})|^{2}\mathrm{d}\mathbf{x}=(2\pi)^{4s} \det b \int_{\mathbb{R}^{n}}b(\xi,\xi)^{2s} \cdot |(\mathcal{F}^{L/R}_{b}f)(\xi)|^{2}\mathrm{d}\xi.
\end{equation*}
\end{itemize}
\end{remark}
In the following result we provide some natural properties of the geometric fractional Laplacian, similar to those of the classical fractional Laplacian.
\begin{proposition}
Let $b$ be a positive definite geometric structure on $\mathbb{R}^n$ and $s\in (0,1)$. Then for any $f,g\in\mathcal{S}(\mathbb{R}^{n})$, $s,t\in (0,1)$ (with $s+t<1$), $\mu,\nu\in\mathbb{R}$, $\alpha\in{\mathbb{N}^{n}}$, $\mathbf{h}\in\mathbb{R}^{n}$, and $\lambda >0$ we have
\begin{itemize}
\item[(i)] $(-\Delta_{b})^{s}(\mu f+\nu g)=\mu~(-\Delta_{b})^{s}f+\nu~(-\Delta_{b})^{s}g$,
\item[(ii)] $(-\Delta_{b})^{s}((-\Delta_{b})^{t}f)=(-\Delta_{b})^{t}((-\Delta_{b})^{s}f)=(-\Delta_{b})^{t+s}f$,
\item[(iii)] $D^{\alpha}((-\Delta_{b})^{s}f)=(-\Delta_{b})^{s}(D^{\alpha}f)$, and consequently $(-\Delta_{b})^{s}f\in\mathcal{C}^{\infty}(\mathbb{R}^{n},\mathbb{C})$,
\item[(iv)] $(-\Delta_{b})^{s}(T_{\mathbf{h}}f)=T_{\mathbf{h}}((-\Delta_{b})^{s}f)$,
\item[(v)] $(-\Delta_{b})^{s}(\delta_{\lambda}f)=\lambda^{2s}\delta_{\lambda}((-\Delta_{b})^{s}f)$,
\end{itemize}
where $(T_{\mathbf{h}}f)(\mathbf{x}):=f(\mathbf{x}+\mathbf{h}), ~\forall\mathbf{x}\in\mathbb{R}^{n}$, and $(\delta_{\lambda}f)(\mathbf{x}):=f(\lambda\mathbf{x}), \forall\mathbf{x}\in\mathbb{R}^{n}$.
\end{proposition}
\begin{proof} Let us start by pointing out that when appropriate, it is sufficient to prove the identities on the (left/right) Fourier transform side. Recall that $g^{s}_{b}:\mathbb{R}^{n}\rightarrow\mathbb{R}$ is the function given by $g^{s}_{b}(\xi):=[4\pi^2 b(\xi,\xi)]^{s}, ~\forall \xi\in\mathbb{R}^{n}$.
\begin{itemize}
\item[(i)] The identity follows directly from the properties of the (left/right) Fourier transforms. More exactly, for any $f,g\in\mathcal{S}(\mathbb{R}^{n})$ and $\mu,\nu\in\mathbb{R}$
\begin{align*}
\mathcal{F}^{L/R}_{b}((-\Delta_{b})^{s}(\mu f+\nu g))&=g^{s}_{b} \cdot \mathcal{F}^{L/R}_{b} (\mu f+\nu g)=\mu ~ (g^{s}_{b} \cdot \mathcal{F}^{L/R}_{b} f)+\nu ~(g^{s}_{b} \cdot \mathcal{F}^{L/R}_{b} g)\\
&=\mu ~\mathcal{F}^{L/R}_{b}((-\Delta_{b})^{s}f)+\nu ~\mathcal{F}^{L/R}_{b}((-\Delta_{b})^{s}g)\\
&=\mathcal{F}^{L/R}_{b}(\mu ~(-\Delta_{b})^{s}f + \nu ~(-\Delta_{b})^{s}g).
\end{align*}
\item[(ii)] For any $f\in\mathcal{S}(\mathbb{R}^{n})$ and $s,t\in (0,1)$ (with $s+t<1$) we have
\begin{align*}
\mathcal{F}^{L/R}_{b}((-\Delta_{b})^{s}((-\Delta_{b})^{t}f))&=g^{s}_{b} \cdot \mathcal{F}^{L/R}_{b}((-\Delta_{b})^{t}f)=g^{s}_{b} \cdot (g^{t}_{b} \cdot \mathcal{F}^{L/R}_{b}f)\\
&=g^{s+t}_{b} \cdot \mathcal{F}^{L/R}_{b}f=\mathcal{F}^{L/R}_{b}((-\Delta_{b})^{s+t}f).
\end{align*}
Interchanging $s$ and $t$ in the previous relation it follows that 
\begin{align*}
\mathcal{F}^{L/R}_{b}((-\Delta_{b})^{t}((-\Delta_{b})^{s}f))=\mathcal{F}^{L/R}_{b}((-\Delta_{b})^{t+s}f)=\mathcal{F}^{L/R}_{b}((-\Delta_{b})^{s+t}f),
\end{align*}
and thus we get the conclusion.
\item[(iii)] Note that it is enough to prove that for any $\alpha\in\mathbb{N}$ and $j\in\{1,\dots,n\}$
\begin{equation}\label{foe}
\partial^{\alpha}_{j}((-\Delta_{b})^{s}f)=(-\Delta_{b})^{s}(\partial^{\alpha}_{j}f).
\end{equation}
Using the items $(i)$ and $(v)$ from Proposition \ref{simpor} (adapted to the appropriate natural extensions of the geometric left/right Fourier transforms), for any $f\in\mathcal{S}(\mathbb{R}^{n})$, $\alpha\in\mathbb{N}$ and $j\in\{1,\dots,n\}$ we obtain successively  
\begin{align*}
\partial^{\alpha}_{j}((-\Delta_{b})^{s}f)&=\partial^{\alpha}_{j}((-\Delta_{b})_{L}^{s}f)=\partial^{\alpha}_{j}[(\mathcal{F}^{L}_{b})^{-1}(\mathcal{F}^{L}_{b}(-\Delta_{b})_{L}^{s}f)]\\
&=(2\pi i)^{\alpha}(\mathcal{F}^{L}_{b})^{-1}((\tau_{B^{\top}}\pi_{j})^{\alpha}\cdot\mathcal{F}^{L}_{b}(-\Delta_{b})_{L}^{s}f)\\
&=(2\pi i)^{\alpha}(\mathcal{F}^{L}_{b})^{-1}((\tau_{B^{\top}}\pi_{j})^{\alpha}\cdot g^{s}_{b}\cdot\mathcal{F}^{L}_{b}f)\\
&=(\mathcal{F}^{L}_{b})^{-1}(g^{s}_{b}\cdot((2\pi i \cdot\tau_{B^{\top}}\pi_{j})^{\alpha}\cdot\mathcal{F}^{L}_{b}f))=(\mathcal{F}^{L}_{b})^{-1}(g^{s}_{b}\cdot \mathcal{F}^{L}_{b}(\partial^{\alpha}_{j}f))\\
&=(\mathcal{F}^{L}_{b})^{-1}(\mathcal{F}^{L}_{b}((-\Delta_{b})^{s}_{L}(\partial^{\alpha}_{j}f)))=(-\Delta_{b})^{s}_{L}(\partial^{\alpha}_{j}f)=(-\Delta_{b})^{s}(\partial^{\alpha}_{j}f),
\end{align*}
hence we proved the relation \eqref{foe}.
\item[(iv)] Let $f\in\mathcal{S}(\mathbb{R}^{n})$ and $\mathbf{h}\in\mathbb{R}^{n}$ be given. Then for any $\xi\in\mathbb{R}^{n}$, using Proposition \ref{transF} $(i)$ we have
\begin{align*}
[\mathcal{F}^{L}_{b}((-\Delta_{b})^{s}(T_{\mathbf{h}}f))](\xi)&=g^{s}_{b}(\xi)\cdot(\mathcal{F}^{L}_{b}(T_{\mathbf{h}}f))(\xi)=g^{s}_{b}(\xi)\cdot (\mathcal{F}^{L}_{b}f)(\xi)\cdot e^{2\pi i b(\xi,\mathbf{h})}\\
&=(\mathcal{F}^{L}_{b}((-\Delta_{b})^{s}f))(\xi)\cdot e^{2\pi i b(\xi,\mathbf{h})}=[\mathcal{F}^{L}_{b}(T_{\mathbf{h}}((-\Delta_{b})^{s}f))](\xi),
\end{align*}
thus we get the conclusion.
\item[(v)] Let $f\in\mathcal{S}(\mathbb{R}^{n})$ and $\lambda>0$ be given. Using the identity $g^{s}_{b}(\xi)=\lambda^{2s}\cdot g^{s}_{b}(\xi/\lambda), ~\forall\xi\in\mathbb{R}^{n}$, and Proposition \ref{transF} $(v)$, for any $\xi\in\mathbb{R}^{n}$ we have
\begin{align*}
[\mathcal{F}^{L/R}_{b}((-\Delta_{b})^{s}(\delta_{\lambda}f))](\xi) &= g^{s}_{b}(\xi)\cdot (\mathcal{F}^{L}_{b}(\delta_{\lambda}f))(\xi)=g^{s}_{b}(\xi)\cdot\lambda^{-n}\cdot(\mathcal{F}^{L/R}f)(\xi/\lambda)\\
&=\lambda^{2s}\cdot g^{s}_{b}(\xi/\lambda)\cdot\lambda^{-n}\cdot(\mathcal{F}^{L/R}f)(\xi/\lambda)\\
&=\lambda^{2s}\cdot\lambda^{-n}\cdot g^{s}_{b}(\xi/\lambda)\cdot(\mathcal{F}^{L/R}f)(\xi/\lambda)\\
&=\lambda^{2s}\cdot\lambda^{-n}\cdot (\mathcal{F}^{L/R}_{b}((-\Delta_{b})^{s}f))(\xi/\lambda)\\
&=\lambda^{2s}\mathcal{F}^{L/R}_{b}(\delta_{\lambda}((-\Delta_{b})^{s}f))(\xi)=[\mathcal{F}^{L/R}_{b}(\lambda^{2s}\delta_{\lambda}((-\Delta_{b})^{s}f))](\xi),
\end{align*}
hence we obtained the desired relation.
\end{itemize}
\end{proof}

Let us now show the geometric character of the fractional geometric Laplacian associated to a positive definite geometric structure $b$, by emphasising its compatibility with the group $G_b$, defined by \eqref{Gb}, and consisting of those endomorphisms of $\mathbb{R}^{n}$ which preserve the geometric structure $b$.
\begin{theorem}
Let $b$ be a positive definite geometric structure on $\mathbb{R}^{n}$ and $s\in (0,1)$. Then for any $f\in\mathcal{S}(\mathbb{R}^{n})$ and $A\in G_b$
\begin{equation*}
(-\Delta_b)^{s}(\tau_A f)=\tau_{A}((-\Delta_b)^{s}f).
\end{equation*}
\end{theorem}
\begin{proof}
For any fixed $s\in (0,1)$, we denote by $g^{s}_{b}:\mathbb{R}^{n}\rightarrow\mathbb{R}$ the function given by $g^{s}_{b}(\xi):=[4\pi^2 b(\xi,\xi)]^{s}, ~\forall \xi\in\mathbb{R}^{n}$, and show that 
\begin{equation}\label{invA}
\tau_{A}g^{s}_{b}=g^{s}_{b}, ~\forall A\in G_b.
\end{equation}
Indeed, for any $A\in G_b$ and $\xi\in\mathbb{R}^n$ we have
\begin{align*}
(\tau_{A}g^{s}_{b})(\xi)=g^{s}_{b}(A^{-1}\xi)=[4\pi^2 b(A^{-1}\xi,A^{-1}\xi)]^{s}=[4\pi^2 b(\xi,\xi)]^{s}=g^{s}_{b}(\xi).
\end{align*}
Using the relations \eqref{invA} and \eqref{invFR} we obtain for any $s\in (0,1)$, $A\in G_b$ and $f\in\mathcal{S}(\mathbb{R}^{n})$ 
\begin{align*}
(-\Delta_{b})^{s}(\tau_{A}f)&=(\mathcal{F}^{L/R}_{b})^{-1}(g^{s}_{b}\cdot\mathcal{F}^{L/R}_{b}(\tau_{A}f))=(\mathcal{F}^{L/R}_{b})^{-1}(\tau_{A}g^{s}_{b}\cdot\mathcal{F}^{L/R}_{b}(\tau_{A}f))\\
&=(\mathcal{F}^{L/R}_{b})^{-1}(\tau_{A}g^{s}_{b}\cdot\tau_{A}(\mathcal{F}^{L/R}_{b}f))=(\mathcal{F}^{L/R}_{b})^{-1}(\tau_{A}(g^{s}_{b}\cdot\mathcal{F}^{L/R}_{b}f))\\
&=\tau_{A}((\mathcal{F}^{L/R}_{b})^{-1}(g^{s}_{b}\cdot\mathcal{F}^{L/R}_{b}f))=\tau_{A}((-\Delta_b)^{s}f),
\end{align*}
thus we get the conclusion.
\end{proof}

The following result provides an integration by parts formula for the geometric fractional Laplacian.
\begin{theorem}
Let $b$ be a positive definite geometric structure on $\mathbb{R}^n$ and $s\in (0,1)$. Then for any $f,g\in\mathcal{S}(\mathbb{R}^n)$
\begin{equation*}
\int_{\mathbb{R}^{n}}((-\Delta_{b})^{s}f)(\mathbf{x})g(\mathbf{x})\mathrm{d}\mathbf{x}=\int_{\mathbb{R}^{n}}f(\mathbf{x})((-\Delta_{b})^{s}g)(\mathbf{x})\mathrm{d}\mathbf{x}.
\end{equation*}
\end{theorem}
\begin{proof}
For any fixed $s\in (0,1)$, let us denote by $g^{s}_{b}:\mathbb{R}^{n}\rightarrow\mathbb{R}$ the function given by $g^{s}_{b}(\xi):=[4\pi^2 b(\xi,\xi)]^{s}, ~\forall \xi\in\mathbb{R}^{n}$. Note that $g^{s}_{b}(-\xi)=g^{s}_{b}(\xi), ~\forall \xi\in\mathbb{R}^{n}$, i.e., $\tau_{-I_d}g^{s}_{b}=g^{s}_{b}$. Then using the formula $(i)$ from Proposition \ref{L2} and $(vi)$ from Theorem \ref{teo2}, adapted to the appropriate natural extensions of the geometric left/right Fourier transforms, the following equalities hold

\begin{align*}
&\int_{\mathbb{R}^{n}}((-\Delta_{b})^{s}f)(\mathbf{x})g(\mathbf{x})\mathrm{d}\mathbf{x}=\langle (\mathcal{F}^{L}_{b})^{-1}(g^{s}_{b}\cdot \mathcal{F}^{L}_{b}f),g\rangle =\langle (\mathcal{F}^{L}_{b})^{-1}(g^{s}_{b}\cdot \mathcal{F}^{L}_{b}f),\mathcal{F}^{R}_{b}((\mathcal{F}^{R}_{b})^{-1}g)\rangle\\
&=\langle \mathcal{F}^{L}_{b}((\mathcal{F}^{L}_{b})^{-1}(g^{s}_{b}\cdot \mathcal{F}^{L}_{b}f)),(\mathcal{F}^{R}_{b})^{-1}g\rangle =\langle g^{s}_{b}\cdot \mathcal{F}^{L}_{b}f,|\det b|\cdot(\tau_{-I_d}\circ\mathcal{F}^{L}_{b})g\rangle \\
&=|\det b|\cdot\langle g^{s}_{b}\cdot \mathcal{F}^{L}_{b}f,(\tau_{-I_d}\circ\mathcal{F}^{L}_{b})g\rangle =|\det b|\cdot\int_{\mathbb{R}^{n}}g^{s}_{b}(\xi)(\mathcal{F}^{L}_{b}f)(\xi)(\mathcal{F}^{L}_{b}g)(-\xi)\mathrm{d}\xi\\
&= |\det b|\cdot\int_{\mathbb{R}^{n}}g^{s}_{b}(-\eta)(\mathcal{F}^{L}_{b}f)(-\eta)(\mathcal{F}^{L}_{b}g)(\eta)\mathrm{d}\eta=|\det b|\cdot\int_{\mathbb{R}^{n}}g^{s}_{b}(\eta)(\mathcal{F}^{L}_{b}g)(\eta)(\mathcal{F}^{L}_{b}f)(-\eta)\mathrm{d}\eta\\
&=|\det b|\cdot\langle g^{s}_{b}\cdot \mathcal{F}^{L}_{b}g,(\tau_{-I_d}\circ\mathcal{F}^{L}_{b})f\rangle=\int_{\mathbb{R}^{n}}f(\mathbf{x})((-\Delta_{b})^{s}g)(\mathbf{x})\mathrm{d}\mathbf{x},
\end{align*}
and thus we get the conclusion.
\end{proof}

We conclude this section by pointing out that to any positive definite geometric structure $b$ on $\mathbb{R}^{n}$, and real number $s\in (0,1)$, one can associate a geometric fractional Laplacian, $(-\Delta)^{s}_{b}$, with basic properties similar to those of the classical fractional Laplacian (for a short survey regarding the classical fractional Laplacian, see, e.g., \cite{DL}).


\bigskip
\bigskip

\noindent {\sc R.M. Tudoran}\\
West University of Timi\c soara\\
Faculty of Mathematics and Computer Science\\
Department of Mathematics\\
Blvd. Vasile P\^arvan, No. 4\\
300223 - Timi\c soara, Rom\^ania.\\
E-mail: {\sf razvan.tudoran@e-uvt.ro}\\
\medskip

\end{document}